\documentclass{article}
\usepackage[utf8]{inputenc}
\usepackage[hidelinks]{hyperref}
\usepackage[misc]{ifsym}
\newcommand{\institute}[1]{\\{\scriptsize
   \begin{tabular}[t]{@{\footnotesize}c@{}}#1\end{tabular}}}
\newcommand{\email}[1]{\\{\scriptsize\tt #1}}
\title{The Weihrauch lattice at the level of $\PiCA$:\\ the Cantor-Bendixson theorem}
\author{
\small{Vittorio Cipriani$^{1,2}$}
\email{vittorio.cipriani17@gmail.com}
\and
\small{Alberto Marcone$^1$}
\email{alberto.marcone@uniud.it}
\and
\small{Manlio Valenti$^{1,3}$}
\email{manliovalenti@gmail.com}
\and
\vspace{-0.5cm}
 \institute{$^1$Dipartimento di Scienze Matematiche, Informatiche e Fisiche, Universit\`a di Udine, Italy\\
 $^2$Current address:  Institute of Discrete Mathematics and Geometry, Technische Universit{\"a}t Wien, Austria}
 \institute{$^3$Department of Mathematics, University of Wisconsin - Madison, USA\\
 Current address: Department of Computer Science, Swansea University, UK}
}
\date{}
\usepackage[margin=3.2cm]{geometry}

\usepackage{amssymb}
\usepackage{amsmath}
\usepackage{amsthm}
\usepackage{float}
\usepackage[shortlabels]{enumitem}
\usepackage{graphicx}
\usepackage{theoremref}
\usepackage{tikz}

\usetikzlibrary{backgrounds}
\usetikzlibrary{arrows}
\usetikzlibrary{arrows.meta}
\usetikzlibrary{shapes,shapes.geometric,shapes.misc}
\usetikzlibrary{fit}
\usetikzlibrary{positioning}
\usetikzlibrary{calc}
\usetikzlibrary{quotes}

\tikzstyle{tikzfig}=[baseline=-0.25em,scale=0.5]
\pgfdeclarelayer{edgelayer}
\pgfdeclarelayer{nodelayer}
\pgfsetlayers{background,nodelayer,edgelayer,main}
\tikzstyle{none}=[inner sep=0mm]
\tikzstyle{every loop}=[]
\tikzset{>={latex[width=1mm,length=1mm]}}
\newcommand{\drawThm}[3]{ \draw #1 %
	\ifx&#2&%
	\else
		node[style=thmref, pos=0.01] {\resizebox{5mm}{3mm}{#2}} %
	\fi
	\ifx&#3&%
	\else
		node[style=thmref, pos=0.99] {\resizebox{5mm}{3mm}{#3}}
	\fi ;
}

\tikzstyle{box}=[fill={rgb,255: red,228; green,228; blue,228}, draw=black, shape=rectangle]

\tikzstyle{reducible}=[->, dashed]
\tikzstyle{strictreducible}=[->]
\tikzstyle{nonreducible}=[->, draw=red]
\tikzstyle{uncomp}=[draw=red, <->]

\tikzstyle{thmref}=[opacity=0,inner sep=2mm]
\usepackage{ifthen}
\usepackage{url}
\usepackage{doi}
\makeatletter
\newcommand{\defas}{:=}
\newcommand{\function}[3]{{#1}:{#2}\rightarrow {#3}}
\newcommand{\partialfunction}[3]{{{#1}} :\subseteq {#2}\rightarrow {#3}}
\newcommand{\multifunction}[3]{{#1}: {#2}\rightrightarrows{#3}}
\newcommand{\partialmultifunction}[3]{{{#1}} :\subseteq {#2} \rightrightarrows {#3}}
\newcommand{\st}{{}\,:\,{}}
\newcommand{\repmap}[1]{\delta_{#1}}
\newcommand{\str}[1]{\langle #1 \rangle}
\newcommand{\incomparable}{~|~}
\newcommand{\length}[1]{\left|{#1}\right|}
\newcommand{\body}[1]{\left[{#1}\right]}
\newcommand{\concat}{^\smallfrown}
\newcommand{\X}{\mathcal{X}}
\newcommand{\Y}{\mathcal{Y}}
\newcommand{\negrepr}[1]{\mathcal{A}({#1})}
\newcommand{\pairing}[1]{\langle #1 \rangle}
\newcommand{\completion}[1]{\overline{#1}}
\newcommand{\dom}{\operatorname{dom}}
\newcommand{\id}{\operatorname{id}}
\newcommand{\Baire}{\mathbb{N}^\mathbb{N}}
\newcommand{\baire}{\mathbb{N}^{<\mathbb{N}}}
\newcommand{\Cantor}{2^\mathbb{N}}
\newcommand{\cantor}{2^{<\mathbb{N}}}
\newcommand{\weireducible}{\le_{\mathrm{W}}}

\newcommand{\strictlyweireducible}{<_\mathrm{W}}
\newcommand{\weiequiv}{\equiv_{\mathrm{W}}}
\newcommand{\weiincomparable}{~|_{\mathrm{W}~}}
\newcommand{\illfounded}{\mathcal{IF}}
\newcommand{\wellfounded}{\mathcal{WF}}
\newcommand{\uniquebranch}{\mathcal{UB}}
\newcommand{\countable}{\mathcal{T}^{\leq\aleph_0}}
\newcommand{\uncountable}{\mathcal{T}^{>\aleph_0}}
\newcommand{\exploded}[1]{\mathsf{Expl}(#1)}

\newcommand{\strongweireducible}{\le_{\mathrm{sW}}}
\newcommand{\strictlystrongweireducible}{<_{\mathrm{sW}}}
\newcommand{\strongweiequiv}{\equiv_{\mathrm{sW}}}
\newcommand{\nats}{\mathbb{N}}
\newcommand{\parallelization}[1]{\widehat{#1}}
\newcommand{\totalization}[1]{\mathsf{T}#1}
\newcommand{\firstOrderPart}[1]{{}^1{#1}}
\newcommand{\ustar}[1]{#1^{u*}}
\newcommand{\boldfaceDelta}{\boldsymbol{\Delta}}
\newcommand{\boldfaceSigma}{\boldsymbol{\Sigma}}
\newcommand{\boldfacePi}{\boldsymbol{\Pi}}
\newcommand{\boldfaceGamma}{\boldsymbol{\Gamma}}
\newcommand{\boldfaceLambda}{\boldsymbol{\Lambda}}
\newcommand{\disjointunion}[2]{\bigsqcup_{{#1}}{#2}}
\newcommand*{\xV}{{\bigsqcup}\kern-0.60em\raisebox{0.6ex}{\tiny $\mathrm{b}$}\kern 0.25em}
\newcommand{\binarydisjointunion}[2]{\xV_{{#1}}{\, #2}}
\newcommand{\lpo}{ \mathsf{LPO} }
\newcommand{\range}{\mathsf{range}}
\newcommand{\mflim}{\mathsf{lim}}
\newcommand{\Choice}[1]{\mathsf{C}_{#1}}
\newcommand{\CBaire}{\mathsf{C}_{\Baire}}
\newcommand{\UCBaire}{\mathsf{UC}_{\Baire}}
\newcommand{\translateCantor}[1][\relax]{\ifx\relax#1 \rho_{\Cantor} \else \rho_{\Cantor}({#1})\fi}
\newcommand{\translateBaire}[1][\relax]{\ifx\relax#1 \rho_{\Baire} \else \rho_{\Baire}({#1})\fi}
\newcommand{\PST}[1][]{\mathsf{PST}_{#1}}
\newcommand{\PTT}[1][]{\mathsf{PTT}_{#1}}
\newcommand{\tree}{\mathbf{Tr}}
\newcommand{\PK}[1][]{\mathsf{PK}_{#1}}
\newcommand{\List}[1][]{\mathsf{List}_{#1}}
\newcommand{\wList}[1][]{\mathsf{wList}_{#1}}
\newcommand{\ScList}[1][]{\mathsf{ScList}_{#1}}
\newcommand{\wScList}[1][]{\mathsf{wScList}_{#1}}
\newcommand{\ScCount}[1][]{\mathsf{ScCount}_{#1}}
\newcommand{\deterministicPart}{\mathsf{Det}}
\newcommand{\CB}[1][]{\mathsf{CB}_{#1}}
\newcommand{\wCB}[1][]{\mathsf{wCB}_{#1}}

\newcommand{\wf}{\mathsf{WF}}
\newcommand{\PiCA}{\boldfacePi^1_1\mathsf{-CA}_0}
\newcommand{\sierpinski}{\mathbb{S}}
\newcommand{\wfsierpinski}{\mathsf{WF}_{\sierpinski}}
\newcommand{\codedChoice}[3]{\ifthenelse{\equal{#1}{}}{\mathsf{C}^{\mathsf{\vphantom{g}#2}}_{{#3}} }{ {#1}\text{-}\mathsf{C}^{\vphantom{g}\mathsf{#2}}_{{#3}} }}
\newcommand{\codedUChoice}[3]{\ifthenelse{\equal{#1}{}}{\mathsf{UC}^{\mathsf{\vphantom{g}#2}}_{{#3}} }{ {#1}\text{-}\mathsf{UC}^{\vphantom{g}\mathsf{#2}}_{{#3}} }}
\newcommand{\sbin}{s_{\mathsf{b}}}
\makeatother

\newtheorem{theorem}{Theorem}[section]
\newtheorem{proposition}[theorem]{Proposition}
\newtheorem{lemma}[theorem]{Lemma}
\newtheorem{corollary}[theorem]{Corollary}
\theoremstyle{definition}
\newtheorem{definition}[theorem]{Definition}
\newtheorem{open}[theorem]{Problem}
\newtheorem{remark}[theorem]{Remark}

\begin{document}
\maketitle
{\let\thefootnote\relax\footnote{{

All three authors’ research was partially supported by the Italian PRIN 2017 Grant \lq\lq Mathematical Logic: models, sets, computability\rq\rq. 

The authors thank the anonymous referee for carefully reading the paper and especially for suggesting isolating the construction now encapsulated in \thref{prop:prefixes}.}}}
\begin{abstract}

    This paper continues the program connecting reverse mathematics and computable analysis via the framework of Weihrauch reducibility. In particular, we consider problems related to perfect subsets of Polish spaces, studying the perfect set theorem, the Cantor-Bendixson theorem and various problems arising from them. In the framework of reverse mathematics, these theorems are equivalent respectively to $\mathsf{ATR}_0$ and $\PiCA$, the two strongest subsystems of second order arithmetic among the so-called big five. As far as we know, this is the first systematic study of problems at the level of $\PiCA$ in the Weihrauch lattice.
    
    We show that the strength of some of the problems we study depends on the topological properties of the Polish space under consideration, while others have the same strength once the space is rich enough.
\end{abstract}
\tableofcontents
\section{Introduction}

In this work, we study the uniform computational strength of theorems arising in classical descriptive set theory related to perfect subsets of Polish spaces. In general, if $P$ is a subset of a topological space $\X$, a point $x \in P$ is a \emph{limit point of $P$}  if for every open set $U$ with $x \in U$ there is a distinct point $y \in P \cap U$: otherwise, we call $x$ \emph{isolated in $P$}. A subset of a topological space is \emph{perfect} if it is closed and has no isolated points. An equivalent formulation is that $P\subseteq \X$ is perfect if $P=P'$, where $P'$ is the Cantor-Bendixson derivative, i.e.\  the set of all limit points of $P$. Notice that every nonempty perfect subset of a Polish space has the cardinality of the continuum.

A classical theorem in this context is the \emph{Cantor-Bendixson theorem}.

\begin{theorem}[Cantor-Bendixson theorem]
\thlabel{Initialtheorem}
  Every closed subset $C$ of a Polish space $\X$ can be uniquely written as the disjoint union of a perfect set $P$ and a countable set $S$. We call $P$ the {perfect kernel} of $C$ and $S$ the {scattered part} of $C$.
\end{theorem}

When $\X$ is either Cantor space $\Cantor$ or Baire space $\Baire$, there is a well-known correspondence between closed sets and sets of \emph{paths} through \emph{trees}. In these settings, the Cantor-Bendixson theorem states that all but countably many paths through a tree $T$ on $\nats$ belong to a perfect subtree $S$ of $T$ (i.e.\ a tree such that every node has two incomparable extensions). Again, $S$ (which is unique) is called the perfect kernel of $T$ and the set of missing paths is the scattered part of $T$.\smallskip

 Many theorems in \lq\lq classical\rq\rq\ mathematics, including \thref{Initialtheorem} and its tree version, can be written in the form 
 \[(\forall x \in X)(\varphi(x)\implies (\exists y \in Y)\psi(x,y)),\]
 and this formulation has a natural translation as a computational problem: given an instance $x \in X$ satisfying $\varphi(x)$, the task is to find a solution $y \in Y$ such that $\psi(x,y)$ (notice that such a $y$, in general, is not unique). A computational problem can be naturally rephrased as a \emph{(partial) multi-valued function} $\partialmultifunction{f}{X}{Y}$ where $f(x)\defas\{y \in Y:\psi(x,y)\}$, for every $x \in X$ such that $\varphi(x)$. The interpretation of theorems as multi-valued functions/problems, allows us to compare their uniform computational content using the framework of \emph{Weihrauch reducibility}.

This work continues the program initiated by Gherardi and the second author \cite{hanhbanach} that aims to provide a bridge between computable analysis and \emph{reverse mathematics}, the discipline which establishes equivalences between mathematical statements and the axioms required to prove them. A well-known empirical fact in this field is the so-called \emph{big five phenomenon}. Namely, many theorems of \lq\lq classical\rq\rq\ mathematics happen to be equivalent to one of five subsystems of second order arithmetic, namely $\mathsf{RCA}_0,\mathsf{WKL}_0,\mathsf{ACA}_0,\mathsf{ATR}_0$ and $\PiCA$. Some analogs of the big five have been identified in the Weihrauch context. For example, $\mathsf{RCA}_0$ roughly corresponds to computable problems, $\mathsf{WKL}_0$ to the task of choosing an element from a nonempty closed subset of the Cantor space (denoted by $\codedChoice{}{}{\Cantor}$) and $\mathsf{ACA}_0$ to iterations of $\mflim$, where $\mflim$ is the problem that takes in input a converging sequence in Baire space and outputs its limit.

The second author in \cite{daghstul} raised the question \lq\lq What do the Weihrauch hierarchies look like once we go to very high levels of reverse mathematics strength?\rq\rq. Here we continue to answer this question by focusing on theorems that, in reverse mathematics, are either equivalent to $\mathsf{ATR}_0$ or $\PiCA$. In this direction, it has been shown that statements that in reverse mathematics are equivalent to $\mathsf{ATR}_0$, when considered as problems, fall into different Weihrauch degrees (see for example \cite{kihara_marcone_pauly_2020, choiceprinciples, openRamsey, computabilitytheoretic, GOH2020102789}). On the other hand, $\PiCA$ has a natural correspondent in the Weihrauch lattice, namely the problem $\parallelization{\wf}$ that given in input a sequence of trees on $\nats$ outputs the sequence $p \in \Cantor$ such that $p(i)=1$ iff the $i$-th tree of the sequence is well-founded. This problem has been briefly considered by Jeff Hirst in \cite{leafmanaegement} and by Goh, Pauly, and the third author in \cite{goh_pauly_valenti_2021}, but ---to the best of our knowledge--- this is the first paper carrying out a systematic study of a theorem equivalent to $\PiCA$.

In \cite{leafmanaegement}, Hirst showed that $\parallelization{\wf}$ is Weihrauch equivalent to the problem of finding the perfect kernel of a tree. We show that for any uncountable computable Polish space $\X$, the problem of finding the perfect kernel of a closed set is strictly below $\parallelization{\wf}$ and its degree does not depend on $\X$. Notice that in reverse mathematics all these problems are equivalent to $\PiCA$ (see \cite[\S VI.1]{simpson_2009}). We consider the full Cantor-Bendixson problem, which, given in input a tree or a closed set, outputs its perfect kernel and its scattered part. Again, the tree version is strictly stronger than the closed set version; however, in this case, the topological properties of $\X$ affect the precise degree of the problem. We also study the problem of just listing the scattered part of a tree or a closed set.

In \cite{kihara_marcone_pauly_2020} Kihara, Pauly and the second author studied different problems related to the perfect tree theorem, i.e.\ the statement \lq\lq a tree on $\nats$ has countably many paths or contains a perfect subtree\rq\rq. The disjunctive nature of the theorem gives rise to two different groups of problems: either the problem takes as input a tree with uncountably many paths and outputs a perfect subtree of it, or it takes a tree with countably many paths and outputs a list of them. Here we study the same two kinds of problems for the perfect set theorem, i.e.\ the statement that a closed subset of a Polish space is either countable or contains a perfect subset. We obtain a number of results for different computable Polish spaces. 

Figures \ref{SummaryAtThebeginning} and \ref{Figureslist} summarize some of our results. The precise definitions of the various functions are given in due time.

\begin{figure}[H]
    \centering
 	\tikzstyle{every picture}=[tikzfig]
		\begin{tikzpicture}
	\begin{pgfonlayer}{nodelayer}
	\node [style=box] (UCBaire) at (-6.5,0) {$\UCBaire$};
	\node [style=box] (ScListCantor) at (-12,0) {$\ScList[\Cantor]$};
	\node [style=box] (PST) at (-6.5,3) {$\PST[\Cantor]\weiequiv\PST[\Baire]$};
	\node [style=box] (CBaire) at (2,6) {$\CBaire\weiequiv \PTT[\Baire]\weiequiv \PTT[\Cantor]$};
	\node [style=box] (PK) at (-12,7) {$\PK[\Baire]\weiequiv\wScList[\Baire]\weiequiv\wCB[\Baire]\weiequiv\CB[\Cantor]$};
	\node [style=box] (ScListBaire) at (-6.5,10) {$\ScList[\Baire]$};
	\node [style=box] (CBBaire) at (-6.5,13) {$\CB[\Baire]$};
	\node [style=box] (PKTree) at (-6.5,16) {$\parallelization{\wf}\weiequiv\PK[]\weiequiv\wCB[]\weiequiv\CB[]$};
	\end{pgfonlayer}
	\begin{pgfonlayer}{edgelayer}
		\draw [style=strictreducible] (UCBaire) to (PST);
		\draw [style=strictreducible] (PST.north) to (CBaire);
		\draw [style=strictreducible] (PST.north) to (PK);
		\draw [style=strictreducible] (PK) to (ScListBaire);
		\draw [style=strictreducible] (CBBaire) to (PKTree);
		\draw [style=strictreducible] (CBaire) to (PKTree);
		\draw [style=strictreducible] (ScListCantor) to (PK);
        \begin{scope}[transform canvas={xshift=-.5em}]
            \draw [->] (CBaire) -- (CBBaire.south east) node[midway,fill=white]{?};
            \draw [->] (CBaire) -- (ScListBaire) node[midway,fill=white]{?};
        \end{scope}
        \draw [->] (CBBaire.290) -- (ScListBaire.70) node[midway,fill=white]{?};
        \draw [strictreducible] (ScListBaire.110) to (CBBaire.250);
	\end{pgfonlayer}
\end{tikzpicture}

		\caption{Some multi-valued functions studied in this paper. The arrows represent Weihrauch reducibility in the direction of the arrow. The question marks indicate that the existence of a reduction is still open. If a function cannot be reached from another one following a path of arrows we know that there is no reduction between the two functions.}
		\label{SummaryAtThebeginning}
\end{figure}
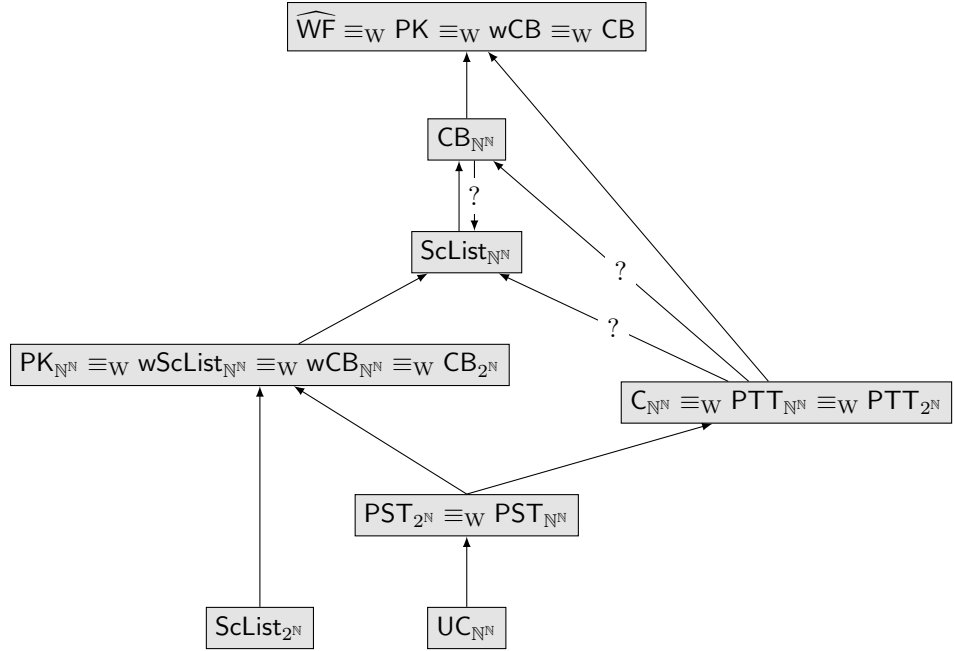

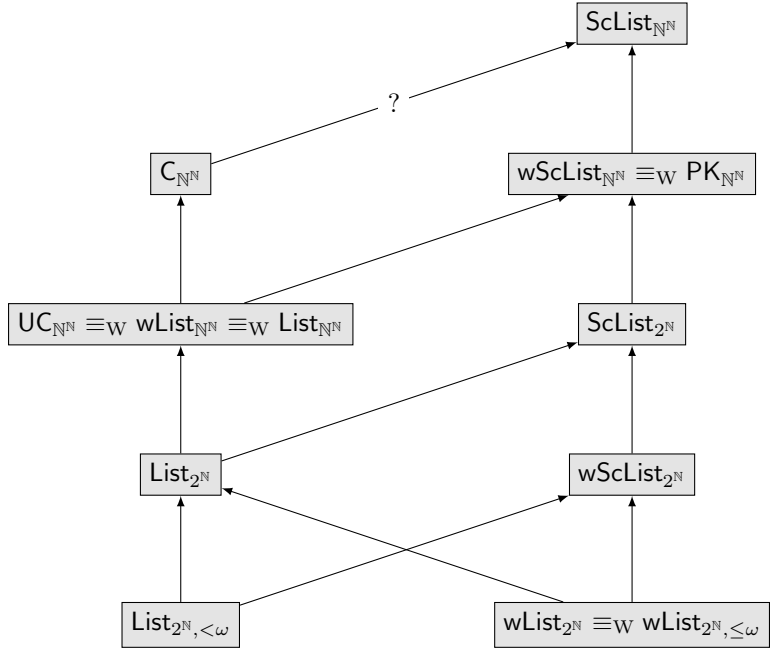
\begin{figure}[H]
	\begin{center}
		\tikzstyle{every picture}=[tikzfig]
\begin{tikzpicture}
	\begin{pgfonlayer}{nodelayer}
	\node [style=box] (1) at (0,0) {$\List[\Cantor,<\omega]$};
	\node [style=box] (2) at (12,0) {$\wList[\Cantor]\weiequiv\wList[\Cantor,\leq \omega]$};
	\node [style=box] (3) at (0,4) {$\List[\Cantor]$};
	\node [style=box] (4) at (12,4) {$\wScList[\Cantor]$};
	\node [style=box] (5) at (0,8) {$\UCBaire \weiequiv\wList[\Baire]\weiequiv\List[\Baire]$};
	\node [style=box] (6) at (12,8) {$\ScList[\Cantor]$};
	\node [style=box] (7) at (0,12) {$\CBaire$};
	\node [style=box] (8) at (12,12) {$\wScList[\Baire]\weiequiv\PK[\Baire]$};
	\node [style=box] (9) at (12,16) {$\ScList[\Baire]$};
	\end{pgfonlayer}
	    \begin{pgfonlayer}{edgelayer}
		\draw [style=strictreducible] (1) to (3);
		\draw [style=strictreducible] (3) to (5);
		\draw [style=strictreducible] (5) to (7);
		\draw [style=strictreducible] (2) to (3);
		\draw [style=strictreducible] (2) to (4);
		\draw [style=strictreducible] (4) to (6);
		\draw [style=strictreducible] (6) to (8);
		\draw [style=strictreducible] (8) to (9);
        \draw [style=strictreducible] (5) to (8);
        \draw [style=strictreducible] (1) to (4);
        \draw [style=strictreducible] (3) to (6);
        \draw [->] (7) -- (9) node[midway,fill=white]{?};
	\end{pgfonlayer}
\end{tikzpicture}
		\caption{Multi-valued functions related to listing problems in the Weihrauch lattice. The arrows have the same meaning of Figure \ref{SummaryAtThebeginning}.}
		\label{Figureslist}
	\end{center}
\end{figure}

The paper is organized as follows. In \S  \ref{Background} we give the necessary preliminaries: namely, in the first part, we provide definitions and notations about trees together with some useful lemmas, while in the second part we deal with represented spaces and Weihrauch reducibility. In \S \ref{perfectsetsgeneral} we study multi-valued functions related to the perfect set and perfect tree theorem in Baire and Cantor space, while in \S \ref{cantorbendixson} we consider problems related to the Cantor-Bendixson theorem in the same setting. In \S \ref{otherspaces} we study the problems considered in \S  \ref{perfectsetsgeneral} and \S \ref{cantorbendixson} for arbitrary computable metric spaces, while \S\ref{Openquestions} lists some open problems that remain to be solved.

\section{Background}
\label{Background}
\subsection{Sequences and trees}
\label{Sequencessandtrees}
Let $\nats^n$ denote the set of finite sequences of natural numbers of length $n$, where the length is denoted by $\length{\cdot}$. If $n=0$, $\nats^0=\{\str{}\}$, where $\str{}$ is the empty sequence: in general, given $i_0,\dots,i_{n-1} \in \nats$, we denote by $\str{i_0,\dots,i_{n-1}}$ the finite sequence in $\nats^n$ having digits $i_0,\dots,i_{n-1}$. The set of all finite sequences of natural numbers is denoted by $\baire$, while we write $\cantor$ for the set of all finite sequences of $0$ and $1$. For $\sigma \in \baire$ and $m\leq\length{\sigma}$, let $\sigma[m]\defas\str{\sigma(0),\dots,\sigma(m-1)}$. Given $\sigma, \tau \in \baire$, we use $\sigma \sqsubseteq \tau$ to say that $\sigma$ is an \emph{initial segment} of $\tau$ (equivalently, $\tau$ an \emph{extension} of $\sigma$), i.e. $\sigma=\tau[m]$ for some $m \leq \length{\sigma}$. We use the symbol $\sqsubset$ in case $\sigma \sqsubseteq \tau$ and $\length{\sigma}<\length{\tau}$, and in case $\sigma \not\sqsubseteq \tau$ and $\tau \not\sqsubseteq \sigma$ we say that $\sigma$ and $\tau$ are \emph{incomparable} ($\sigma \incomparable \tau$). 

The concatenation of two finite sequences $\sigma,\tau$ is denoted by $\sigma\concat\tau$, but often we just write $\sigma\tau$. The same symbol is also used for the concatenation of a finite and an infinite sequence. For $n,k \in \nats$, we denote by $n^k$ the sequence made of $k$ many $n$'s: in case $k=1$ we just write $n$ and we use $n^\nats$ to denote the infinite sequence with constant value $n$. For $\sigma \in \baire$ and $p \in \Baire$ we denote by $\sigma^-$ and $p^-$ the result of deleting the first digit of the sequence.

\begin{remark}
\thlabel{Bijections}
We fix a bijection between $\baire$ and $\nats$: to avoid too much notation we do not introduce a specific symbol for this bijection, but we identify a sequence with the number representing it: it should be clear from the context whether we are referring to a finite sequence or to the number representing it. We need the bijection to enjoy all the usual properties such as $\sigma \mapsto \length{\sigma}$ being computable: moreover, we require that if $\sigma \sqsubset \tau$ then $\sigma<\tau$. 
\end{remark}

A tree $T$ is a nonempty subset of $\baire$ closed under initial segments. In case the tree $T$ is a subset of $\cantor$, we  call $T$ a binary tree. We say that $f \in \Baire$ is a path through $T$ if for all $n \in \nats$, $f[n]\in T$ where, as for finite sequences, $f[n]=\str{f(0),\dots,f(n-1)}$. We denote by $\body{T}$ the \emph{body} of $T$, that is the set of paths through $T$. We say that a tree $T$ is \textit{ill-founded} iff there exists at least one path in $\body{T}$ and \textit{well-founded} otherwise. Given $\sigma \in T$ we define the tree of extensions of $\sigma$ in $T$ as $T_{\sigma}\defas\{\tau:\tau \sqsubseteq \sigma \lor \tau \sqsupseteq \sigma\}$. Notice that $T_\sigma$ is ill-founded iff there exists a path through $T$ that extends $\sigma$. We say that $T$ is \emph{perfect} if every element of $T$ has (at least) two incompatible extensions in $T$, that is, $(\forall \sigma \in T)(\exists \tau,\tau' \in T)(\sigma \sqsubset \tau \land \sigma \sqsubset \tau' \land \tau \incomparable \tau')$. It is straightforward that the body of a nonempty perfect tree has uncountably many paths. Given a tree $T$, the largest perfect subtree $S$ of $T$ is called the \emph{perfect kernel} of $T$ while $\body{T} \setminus \body{S} \subseteq \Baire$ is called the \emph{scattered part} of $T$. We call $T$ \emph{pruned} if every $\sigma \in T$ has a proper extension. Notice that every perfect tree is pruned. Moreover, if $\body{T}$ is perfect and $T$ is pruned then $T$ is a perfect tree.
\begin{remark}
\thlabel{Perfectnessincantor}
It is useful to notice that, for a binary tree $T$, if $\length{\body{T}}>\aleph_0$ then there must uncountably many paths with infinitely many ones. In other words, it can't be the case that all the paths in $\body{T}$ are eventually  zero paths, as it is straightforward to notice that they are only countably many.
\end{remark}

Given trees $T$ and $S$, we define the \textit{disjoint union} of $T$ and $S$ as $T\sqcup S=\{\str{}\} \cup \{\str{0} \tau: \tau \in  T \} \cup \{\str{1} \tau: \tau \in  S \}$. Of course, this is still a tree and it has the property that $T\sqcup S$ is ill-founded iff at least one of $T$ and $S$ is ill-founded. The construction can be easily generalized to countably many trees letting $\disjointunion{i\in \nats}{T^i}\defas\{\str{}\} \cup \{\str{i} \tau:\tau \in T^i \land i \in \nats\}$ and we still have that $\disjointunion{i\in \nats}{T^i}$ is ill-founded iff there exists $i$ such that $T^i$ is ill-founded. We also define the \emph{binary disjoint union} as  $\binarydisjointunion{i\in \nats}{T^i}\defas\{\str{}\} \cup \{0^i\str{1}\tau: \tau \in T^i \land i \in \nats\}$

\begin{remark}
\thlabel{Disjoint_union}
Notice that if all the $T^i$'s are binary trees, also $\binarydisjointunion{i\in \nats}{T^i}$ is and, regardless the ill-foundedness/well-foundedness of the $T^i$'s, $0^\nats\in\body{\binarydisjointunion{i\in \nats}{T^i}}$. Moreover $\length{\body{\binarydisjointunion{i\in \nats}{T^i}}}=1+\sum\limits_{i \in \nats} \length{\body{T^i}}$. In particular, $\length{\body{\binarydisjointunion{i\in \nats}{T^i}}}>1$  iff there exists an $i \in \nats$ such that $T^i$ is ill-founded. 
\end{remark}

We now turn our attention to another operation on trees, namely \emph{interleaving}. Given $\sigma,\tau \in \nats^{n}$, we define $\sigma*\tau\defas\str{ \sigma(0),\tau(0),\dots ,\sigma(n-1),\tau(n-1) }$. The same definition applies to infinite sequences. Then given trees $T$ and $S$, the {interleaving} between $T$ and $S$ is $T*S\defas\{ \sigma*\tau: \length{\sigma}=\length{\tau} \land \sigma \in T \land \tau \in S\}$. Clearly, $T*S$ is a tree and it is ill-founded iff both $T$ and $S$ are ill-founded. This construction can be generalized to countably many trees in a straightforward way and we use a notation such as $\underset{i \in \nats}{*}T^i$.

We often use the interleaving $\exploded{T}\defas T*\cantor$, which we call the \emph{explosion} of $T$.

Sometimes it is useful to be able to \lq\lq translate\rq\rq\ back and forth between sequences of natural numbers and binary sequences. 

\begin{definition}
\thlabel{Translationfinitesequences}
We define:
\begin{itemize}
    \item $\translateCantor\colon\baire \rightarrow \cantor$ by $\translateCantor(\sigma)\defas0^{\sigma(0)}10^{\sigma(1)}1\dots10^{\length{\sigma}-1}1$; in particular, $\translateCantor(\str{})\defas\str{}$;
    \item $\translateBaire\colon\cantor\rightarrow \baire$ by
    \[\translateBaire(\tau)\defas
    \begin{cases}
    \translateCantor^{-1}({\tau[n_\tau+1]})&\text{if } (\exists i)(\tau(i)=1) \text{ where } n_\tau\defas\max \{i:\tau(i)=1\};
    \\
    \str{} & \text{if } (\forall i)(\tau(i)=0).
    \end{cases}\]

\end{itemize} 
\end{definition}

The two functions defined above have the following properties:
\begin{itemize}
    \item $\translateCantor$ is injective;
    \item  $\translateBaire(\translateCantor(\sigma))=\sigma$;
    \item $\sigma \sqsubset \sigma'$ iff $\translateCantor(\sigma) \sqsubset \translateCantor(\sigma')$;
    \item if $\tau \sqsubset \tau'$ then  $\translateBaire(\tau) \sqsubseteq \translateBaire(\tau')$. 
\end{itemize}

We are now able to \lq\lq translate\rq\rq\ back and forth between trees on $\nats$ and binary trees. We use the same symbols $\translateCantor$ and $\translateBaire$ as the context explains which function we are using. 
\begin{definition}
\thlabel{Translatetrees}
Let $T\subseteq \cantor$ and $S\subseteq \baire$ be trees. We define:
\begin{itemize}
    \item    $\translateBaire(T)\defas\{\sigma \in \baire: \translateCantor(\sigma) \in T\}$;
     \item    $\translateCantor(S)\defas\{\tau \in \cantor:\translateBaire(\tau)\in S\}$.
\end{itemize}
\end{definition}

\begin{remark}
\thlabel{Translation}
Notice that, since $\translateCantor(\sigma0^n)=\translateCantor(\sigma)$ for every $n$, if $\sigma \in \translateBaire(T)$ then $\sigma0^\nats \in \body{\translateBaire(T)}$. It is straightforward to check that $\translateBaire(T)=\{\translateBaire(\tau) \in \baire: \tau \in T\}$. On the other hand, for most trees, $S\subseteq \baire$, $\translateCantor(S) \neq  \{\translateCantor(\tau) \in \cantor: \tau \in S\}$ as the latter is not even a tree.
\end{remark}

The back and forth translations between sequences in $\Baire$ and $\Cantor$ are also denoted by the same function symbols used for finite sequences and for trees: again the context clarifies which one we are using. 
\begin{definition}
\thlabel{translateinfinitesequence}

\begin{itemize}
    \item $\translateCantor\colon\Baire \rightarrow \Cantor$ is defined by $\translateCantor(p)\defas\underset{n \in \nats}{\bigcup}\translateCantor(p[n])=0^{p(0)}10^{p(1)}\dots 10^{p(n)}1\dots$;
        \item $\partialfunction{\translateBaire}{\Cantor}{\Baire}$ has domain  $\{q:(\exists^\infty i)(q(i)=1)\}$ and is defined by $\translateBaire(q)\defas\underset{n \in \nats}{\bigcup}\translateBaire(q[n])$.
\end{itemize}
In both definitions, the union makes sense because the finite sequences are comparable.
\end{definition}

Notice that all the functions $\translateBaire$ and $\translateCantor$ we defined are computable. For the functions on finite sequences, this means usual Turing computability. For the functions on trees and infinite sequences computability is to be intended in the sense of \S \ref{representedspaces}.
\begin{lemma}
\thlabel{AllPropertiesOfTranslation}
The following lemma summarizes the fundamental properties of $\translateCantor$ and $\translateBaire$ for infinite sequences and trees.
\begin{enumerate}
    \item The range of $\translateCantor$ is $\{q \in \Cantor:(\exists^\infty i)(q(i)=1)\}$;
    \item $\translateBaire(\translateCantor(p))=p$ for every $p \in \Baire$.
     \item $\translateCantor(\translateBaire(q))=q$ for every $q \in \dom(\translateBaire)$.
     \item If $S\subseteq \baire$, $p \in \body{S} \iff \translateCantor(p)\in \body{\translateCantor(S)}$ and hence $\body{\translateCantor(S)}\subseteq \{\translateCantor(p): p \in \body{S}\}\cup \{q:(\forall^{\infty} i)(q(i)=0)\}$ so that $\length{\body{\translateCantor(S)}}\leq \aleph_0 \iff \length{\body{S}} \leq \aleph_0$.
     \item If $T\subseteq \cantor$ and $q \in \dom(\translateBaire)$ we have that $q \in \body{T} \iff \translateBaire(q) \in \body{\translateBaire(T)}$.
     \item  If $T\subseteq \cantor$ and $p \in \Baire$ then $p \in \body{\translateBaire(T)} \iff \translateCantor(p) \in \body{T}$.
\end{enumerate}
\end{lemma}
\begin{proof}
The proofs are straightforward from the definitions above.
\end{proof}

\begin{lemma}
\thlabel{Perfecttreesinbaire}
    If $T$ is a binary tree is such that $\body{T}$ is perfect then $\body{\translateBaire(T)}$ is perfect as well. Furthermore, if $T$ is a perfect tree then $\translateBaire(T)$ is a perfect tree as well.
\end{lemma}
\begin{proof}
Let $T$ be a binary tree such that $\body{T}$ is perfect. We show that no $f \in \body{\translateBaire(T)}$ is isolated, i.e.\ $(\forall n)(\exists g \in \body{\translateBaire(T)})(f[n] \sqsubset g \land f \neq g)$.
 Fix $n$: by \thref{AllPropertiesOfTranslation}(6) we get that $\translateCantor(f) \in \body{T}$ and, in particular, $\sigma \defas \translateCantor(f[n]) \in T$. Since $\body{T}$ is perfect, by \thref{Perfectnessincantor}, there exists $h \in \body{T}$ with infinitely many ones such that $\sigma \sqsubset h$ and $\translateCantor(f) \neq h$. By \thref{AllPropertiesOfTranslation}(5) $\translateBaire(h) \in \body{\translateBaire(T)}$ and letting $g\defas \translateBaire(h)$ we reach the conclusion.
 
In case $T$ is a perfect tree, it suffices to show that $\translateBaire(T)$ is pruned. Suppose there exists $\sigma \in \translateBaire(T)$ with no extensions in $\translateBaire(T)$. Then $\tau\defas \translateCantor(\sigma)$ belongs to $T$ and the only path in $T$ extending $\tau$ is of the form $\tau 0^\nats$, contradicting the perfectness of $T$.
\end{proof}
Notice that if $S\subseteq \baire$ is a perfect tree it may be the case that $\body{\translateCantor(S)}$ is not perfect, e.g.\ let $S=\{\sigma \in \baire:\sigma(0)=0\}$ and notice that $0^\nats$ is isolated in $\body{\translateCantor(S)}$.

\subsection{Represented spaces and Weihrauch reducibility}
\label{representedspaces}
We give a brief introduction to computable analysis and the theory of represented spaces. We refer the reader to \cite{Weihrauch} and \cite{brattka2021weihrauch} for more on these topics. In particular, we assume the notion of partial computable function from $\Baire$ to $\Baire$, as formulated in the so-called TTE theory of computation. We denote by $\Phi_e$ the partial function computed by the Turing machine of index $e \in \nats$.

A \emph{represented space} $\textbf{X}$ is a pair $(X,\repmap{X})$  where $X$ is a set and $\partialfunction{\repmap{X}}{\Baire}{X}$ is a (possibly partial) surjection. For each $x \in X$ we say that $p$ is a $\repmap{X}$-name for $x$ if $\repmap{X}(p)=x$.

A computational problem $f$ between represented spaces $\textbf{X}$ and $\textbf{Y}$ is formalized as a \emph{partial multi-valued function} $\partialmultifunction{f}{\textbf{X}}{\textbf{Y}}$. A (possibly partial) function $\partialfunction{F}{\Baire}{\Baire}$ is a \emph{realizer} for $\partialmultifunction{f}{\textbf{X}}{\textbf{Y}}$ if for every $p \in \dom(f\circ \delta_X)$, $\delta_Y(F(p)) \in f(\delta_X(p))$. The notion of realizer allows us to transfer properties of functions on the Baire space (such as computability or continuity) to multi-valued functions defined on represented spaces in general. Whenever we say that a multi-valued function between represented spaces is computable we mean that it has a computable realizer.

We fix a computable enumeration $(q_i)_{i \in \nats}$  of  $\mathbb{Q}$ and we represent the space $\mathbb{R}$ by the so-called Cauchy representation $\delta_\mathbb{R}$ where $\dom(\delta_\mathbb{R})\defas\{p \in \Baire:(\forall j)(\forall i>j)(|q_{p(i)}-q_{p(j)}|<2^{-j})\}$ and $\delta_\mathbb{R}(p) \defas \lim q_{p(n)}$. 

We can now represent any \emph{computable metric space} $\X=(X,d,\alpha)$, that is a separable metric space $(X,d)$ and a dense sequence $\function{\alpha}{\nats}{X}$ such that $\function{d\circ (\alpha \times \alpha)}{\nats^2}{\mathbb{R}}$ is a computable double sequence of real numbers. The Cauchy representation $\partialfunction{\delta_X}{\Baire}{X}$ of such a space has domain $\{p \in \Baire:(\forall j)(\forall i>j)(d(\alpha(p(i)),\alpha(p(j)))<2^{-j})\}$ and is defined by  $\delta_X(p)=\lim \alpha(p(n))$. We always assume that computable metric spaces are represented by this representation. For convenience, we fix a computable enumeration $(B_i)_{i \in \nats}$ of all basic open sets of $\X$,  where the ball $B_{\pairing{n,m}}$ is centered in $\alpha(n)$ and has radius $q_m$.

A computable Polish space is a computable metric space $\X=(X,d,\alpha)$ such that the metric $d$ is complete.

Given a computable metric space $\X$ and $k>0$, using the inductive definition of Borel sets, we can define the represented spaces $\boldfaceSigma_k^0(\X)$, $\boldfacePi_k^0(\X)$ and $\boldfaceDelta_k^0(\X)$: this shows that the Borel classes can be naturally considered as represented spaces.

\begin{definition}[{\cite[Definition 3.1]{effectiveborelmeasurability}}]
\thlabel{Borelrepspaces}
For any computable metric space $\X=(X,d,\alpha)$ and for any $k>0$,  we define the represented spaces $(\boldfaceSigma_k^0(\X),\delta_{\boldfaceSigma_k^0(\X)})$, $(\boldfacePi_k^0(\X),\delta_{\boldfacePi_k^0(\X)})$ and $(\boldfaceDelta_k^0(\X),\delta_{\boldfaceDelta_k^0(\X)})$ inductively as follows: 
\begin{itemize}
\item $\delta_{\boldfaceSigma_1^0(\X)}(p)\defas\underset{i \in \nats}{\bigcup}B_{p(i)}$;
\item $\delta_{\boldfacePi_k^0(\X)}\defas X\setminus\delta_{\boldfaceSigma_k^0(\X)}(p)$;
\item $\delta_{\boldfaceSigma_{k+1}^0(\X)}(p_0*p_1*\dots)\defas\underset{i \in \nats}{\bigcup}\delta_{\boldfacePi_k^0(\X)}(p_i)$.
\end{itemize}
\end{definition}

Notice that the representation of $\boldfacePi_1^0(\X)$ is the standard negative representation of closed sets of $\X$:  as usual in the literature, we denote this represented space by $\negrepr{\X}$. We denote by $\Pi_1^0(\X)$ the collection of closed sets of $\X$ having a computable name.
 
 Let $\tree$ and $\tree_2$ be the represented spaces of trees on $\nats$ numbers and binary trees respectively: for both the representation map is the characteristic function and hence we identify $\tree$ and $\tree_2$ with closed subsets of $\Cantor$. The surjective function $\function{[\cdot]}{\tree}{\negrepr{\Baire}}$ defined by $T\mapsto \body{T}$ is computable with multi-valued computable inverse and the same holds for its restriction to $\tree_2$, which is onto $\negrepr{\Cantor}$. This means that the negative representation of a closed subset $C$ of $\Baire$ (resp.\ $\Cantor$) is equivalent (in the sense of \cite[Definition 2.3.2]{Weihrauch}) to the one given by the characteristic function of a (resp.\ binary) tree $T$ such that $\body{T}=C$. We refer to the latter representation as the \emph{tree representation}.

We can represent the pointclass $\boldfaceSigma_1^1(\X)$ of analytic subsets of $\X$ by defining a name for $S$ as a name for a closed set $C\subseteq \X\times \Baire$ such that $S$ is the projection on the first coordinate of $C$. Then, a name for a coanalytic set $R  \in \boldfacePi_1^1(\X)$ is just a name for its complement.

The next theorem summarizes some well-known results about the level at which some subsets of trees are in the \emph{Kleene arithmetical and analytical hierarchy} (or lightface hierarchy), the effective counterpart of the Borel and projective hierarchy (or boldface hierarchy).  Here, completeness is defined with respect to effective Wadge reducibility. These notions can be found, sometimes with different terminology, for example in \cite{Moschovakis, kechris2012classical}.
\begin{theorem}
\thlabel{Complexityresults}
The following classification results hold:
\begin{enumerate}[(i)]
    \item The set $\illfounded\defas \{T \in \tree : T \text{ is ill-founded}\}$ is $\Sigma_1^1$-complete, while $\wellfounded\defas \{T \in \tree : T \text{ is well-founded}\}$ is $\Pi_1^1$-complete. In contrast, $\illfounded_2\defas \illfounded \cap \tree_2$ is $\Pi_1^0$-complete and $\wellfounded_2\defas \wellfounded \cap \tree_2$ is $\Sigma_1^0$-complete.
    \item The set $\uncountable\defas \{T \in \tree: \length{\body{T}} > \aleph_0\}$  is $\Sigma_1^1$-complete, while $\countable \defas \{T \in \tree: \length{\body{T}} \leq \aleph_0\}$ is $\Pi_1^1$-complete. In this case, $\uncountable_2 \defas \uncountable \cap \tree_2$ is $\Sigma_1^1$-complete as well and $\countable_2 \defas \countable \cap \tree_2$ is also  $\Pi_1^1$-complete.
    \item The set $\uniquebranch\defas \{T \in \tree : \length{\body{T}}=1\}$ is $\Pi_1^1$-complete. In contrast, $\uniquebranch_2 \defas \uniquebranch\cap\tree_2$ is $\Pi_2^0$-complete.
\end{enumerate}
\end{theorem}
\begin{proof}
To show that $\illfounded$ is $\Sigma_1^1$-complete see \cite[Theorem 27.1]{kechris2012classical}: the theorem states the boldface case, but its proof works also in the lightface one. If $T\in \tree_2$, notice that, by K\"{o}nig's lemma, $T \in \illfounded_2$ iff $(\forall n)(\exists \tau \in 2^{n})(\tau \in T)$; hence $\illfounded_2$ is $\Pi_1^0$ and completeness is straightforward. It follows immediately that $\wellfounded$ is $\Pi_1^1$-complete and $\wellfounded_2$ is $\Sigma_1^0$-complete.

To prove that $\uncountable$ is $\Sigma_1^1$-complete notice that, by the Cantor-Bendixson theorem for trees, $T \in \uncountable $ iff $(\exists S \subseteq T)(S \text{ is nonempty and perfect})$: the latter is a $\Sigma_1^1$ formula and it remains to show that $\uncountable$ is complete for $\Sigma_1^1$ sets. This is immediate as $T \in \illfounded \iff \exploded{T} \in \uncountable$. The proof for $\uncountable_2$ is similar and it follows immediately that $\countable$ and $\countable_2$ are $\Pi_1^1$-complete.

To prove that $\uniquebranch$ is $\Pi_1^1$-complete notice that, by the effective perfect set theorem (see \cite[Theorem 4F.1]{Moschovakis}), $T \in \uniquebranch$ iff 
\[(\exists p \in \mathsf{HYP}(\Baire))(p \in \body{T}) \land (\forall \tau,\tau')(\tau \incomparable \tau' \implies T_\tau \in \wellfounded \lor T_{\tau'} \in \wellfounded),\]
where $\mathsf{HYP}(\Baire)$ is the set of hyperarithmetical elements in $\Baire$.
Notice that the formula is $\Pi_1^1$: indeed, the second conjunct is clearly $\Pi_1^1$ and by Kleene's quantification theorem (see \cite[Theorem 4D.3]{Moschovakis}), the first conjunct is $\Pi_1^1$ as well. It remains to show that $\uniquebranch$ is complete for $\Pi_1^1$ sets. To do so, it suffices to notice that $T \in \wellfounded$ iff $S \in \uniquebranch$, where $S\defas  \{0^n:n \in \nats\} \sqcup T$ (indeed, $\body{S}=\{0^\nats\} \cup \{1p:p \in \body{T}\}$). If $T \in \tree_2$, notice that $T \in \uniquebranch_2$ iff 
\[T \in \illfounded_2 \land (\forall \tau,\tau')(\tau \incomparable \tau' \implies T_\tau \in \wellfounded_2 \lor T_{\tau'} \in \wellfounded_2).\]
The formula is clearly $\Pi_2^0$ and proving completeness is straightforward.
\end{proof}

To compare the uniform computational content of different problems, we use the framework of \emph{Weihrauch reducibility}. We say that a problem $f$ is \emph{Weihrauch reducible} to a problem $g$, written $f \weireducible g$, if there are computable maps $\partialfunction{\Phi,\Psi}{\Baire}{\Baire}$ such that if $p$ is a name for some $x \in \dom(f)$, then:
\begin{enumerate}[(i)]
	\item $\Phi(p)$ is a name for some $y \in \dom(g)$;
	\item for every name $q$ for some element of $g(y)$, $\Psi(p*q)$ is a name for some element of $f(x)$.
\end{enumerate}
Informally, if $f\weireducible g$ we are claiming the existence of a procedure for solving $f$ which is computable modulo a single invocation of $g$ as an oracle (in other words, this procedure transforms realizers for $g$ into realizers for $f$). In case $\Phi$ is as above and $\Psi$ is not allowed to use $p$ in its computation, we say that $f$ is \emph{strongly Weihrauch reducible} to a problem $g$, written $f \strongweireducible g$.

Weihrauch reducibility and strong Weihrauch reducibility are reflexive and transitive hence they induce the equivalence relations
$\weiequiv$ and $\strongweiequiv$: that is $f\weiequiv g$ iff $f \weireducible g$ and $g\weireducible f$ (similarly for $\strongweireducible$).
The $\weiequiv$-equivalence classes are called \emph{Weihrauch degrees} (similarly the $\strongweiequiv$-equivalence classes are called  \emph{strong Weihrauch degrees}). Both the Weihrauch degrees and the strong Weihrauch degrees form lattices (see \cite[Theorem\ 3.9\ and\ Theorem\ 3.10]{brattka2021weihrauch}). 

There are several natural operations on problems which also lift to the $\weiequiv$-degrees and the $\strongweiequiv$-degrees: we mention below the ones we need.
 
\begin{itemize}
    \item The \emph{parallel product} $f \times g$ is defined by $(f \times g)(x,y) \defas f(x) \times g(y)$. 
    \item The \emph{finite parallelization} is defined as $f^*((x_i)_{i<n})\defas\{(y_i)_{i<n}:(\forall i<n)(y_i \in f(x_i))\}$.
    \item The \emph{infinite parallelization} is defined as $\parallelization{f}((x_i)_{i \in \nats})\defas\{(y_i)_{i \in \nats}:(\forall i)(y_i \in f(x_i)\}$.
    \end{itemize}
    Informally, the three operators defined above, capture respectively the idea of using $f$ and $g$ in parallel, using $f$ a finite (but given in the input) number of times in parallel, and using $f$ countably many times in parallel. 
    
The following definition, with a slightly different notation, was recently given by Sold\`a and the third author.
    \begin{definition}[\cite{valentisolda}]
    \thlabel{ustar}
    For every $\partialmultifunction{f}{\mathbf{X}}{\mathbf{Y}}$, define the finite unbounded parallelization $\partialmultifunction{\ustar{f}}{\nats \times \Baire \times \mathbf{X}}{(\Baire)^{<\nats}}$ as follows:
    \begin{itemize}
        \item instances are triples $(e,w,(x_n)_{n \in \nats})$ such that $(x_n)_{n \in \nats} \in \dom(\parallelization{f})$ and for each sequence $(q_n)_{n \in \nats}$ with $\repmap{Y}(q_n) \in f(x_n)$, there is a $k\in \nats$ such that $\Phi_e(w,q_0*\dots * q_{k-1})(0)\downarrow$ in $k$ steps;
        \item a solution for $(e,w,(x_n)_{n \in \nats})$ is a finite sequence $(q_n)_{n<k}$ such that for every $n <k$, $\repmap{Y}(q_n)\in f(x_n)$ and  $\Phi_e(w,q_0*\dots * q_{k-1})(0)\downarrow$ in $k$ steps.
    \end{itemize}
    \end{definition}
    Informally, $\ustar{f}$ takes an input a Turing functional with a parameter and an input for $\parallelization{f}$ and outputs \lq\lq sufficiently many\rq\rq\ names for solutions where \lq\lq sufficiently many\rq\rq\ is determined by the convergence of the Turing functional in input.
    
    We call $f$ a \emph{cylinder} if $f \strongweiequiv f \times \id$. If $f$ is a cylinder, then $g \weireducible f$ iff $g \strongweireducible f$ (\cite[Cor.\ 3.6]{BG09}). This is useful for establishing nonreductions because, if $f$ is a cylinder, then it suffices to diagonalize against all strong Weihrauch reductions from $g$ to $f$ in order to show that $g \not\weireducible f$. Cylinders are also useful when working with compositional products (discussed below). Observe that for every problem $f$, $f \times \id$ is a cylinder which is Weihrauch equivalent to $f$.

The \emph{compositional product} $f * g$ captures the idea of what can be achieved by first applying $g$, possibly followed by some computation, and then applying $f$. Formally, $f*g$ is any function satisfying
\[ f * g \weiequiv \max_{\weireducible} \{f_1 \circ g_1 \st f_1 \weireducible f \land g_1 \weireducible g\}. \]
This operator was first introduced in \cite{BolWei11}, and proven to be well-defined in \cite{BP16}. For each problem $f$, we denote by $f^{[n]}$ the $n$-fold iteration of the compositional product of $f$ with itself, i.e., $f^{[1]} = f$, $f^{[2]} = f * f$, and so on.

Many (non) reductions in this paper follow from the characterization of the \emph{first-order} part $\firstOrderPart{f}$ and the \emph{deterministic part} $\deterministicPart(f)$ of a problem $f$. The first was introduced in \cite{dzafarovsolomonyokoyama} and extensively studied in \cite{valentisolda}, and the second one was defined in \cite{goh_pauly_valenti_2021}.

We say that a computational problem $f$ is \emph{first-order} if its codomain is $\nats$. As we need only the following characterization of $\firstOrderPart{f}$ we omit the technical definition (see e.g.\ \cite[Definition 2.2]{goh_pauly_valenti_2021}). 

\begin{theorem}[\cite{dzafarovsolomonyokoyama}]
	For every problem $f$, $\firstOrderPart{f} \weiequiv \max_{\weireducible}\{ g\st g \text{ is first-order and } g \weireducible f\}$.
\end{theorem}

The following theorem relates the first-order part with the unbounded finite parallelization.

\begin{theorem}[{\cite[Theorem 5.7]{valentisolda}}]
\thlabel{Summaryfopustar}
 For every first-order $f$, $\firstOrderPart{(\parallelization{f}) }\weiequiv \ustar{f}$.
\end{theorem}

Similarly, we only need the following characterization of $\deterministicPart(f)$ and we omit the formal definition (see \cite[Definition 3.1]{goh_pauly_valenti_2021}).

\begin{theorem}[{\cite[Theorem 3.2]{goh_pauly_valenti_2021}}]
	For every problem $f$, 
	\[ \deterministicPart(f) \weiequiv \max_{\weireducible}\{g: \partialfunction{g}{X}{\Baire} \land g \weireducible f\}.\]
\end{theorem}

Hence, $\deterministicPart(f)$ is the strongest single-valued function which Weihrauch reduces to $f$.

Other useful operations on problems do not lift to Weihrauch degrees (i.e.\ applying the operation to equivalent problems does not always produce equivalent problems).

The first such operation is the \emph{jump}. First, we define the jump of a represented space $\mathbf{X}=(X,\repmap{X})$. This is $\mathbf{X}'=(X,\repmap{X}')$ where $\repmap{X}'$ takes in input a sequence of elements of $\Baire$ converging to some $p \in \dom(\delta_X)$ and returns $\repmap{X}(p)$. Then for a problem $\partialmultifunction{f}{\textbf{X}}{\textbf{Y}}$ its jump $ \partialmultifunction{f'}{\textbf{X}'}{\textbf{Y}}$ is defined as $f'(x) \defas f(x)$. In other words, $f'$ is the following task: given a sequence that converges to a name for an instance of $f$, produces a solution for that instance. The jump lifts to strong Weihrauch degrees but not to Weihrauch degrees (see \cite[\S 5]{BolWei11}). We use $f^{(n)}$ to denote the $n$-th iterate of the jump applied to $f$.
Let $ \partialfunction{\mflim}{(\Baire)^{\mathbb{N}}}{\Baire}, \ (p_n)_{n \in \mathbb{N}}\mapsto \lim p_n$ be the single-valued function 
whose domain consists of all converging sequences in $\Baire$: we have $f^{(n)}\weireducible f* \mflim^{[n]}$, but the converse reduction does not hold in general. 

We now introduce the \emph{totalization of a problem} and the \emph{completion of a problem}. These two operators are different ways of making a partial multi-valued function total; neither of them lifts to Weihrauch degrees. Given a partial multi-valued function $\partialmultifunction{f}{\mathbf{X}}{\mathbf{Y}}$ the totalization of $f$ is the total multi-valued function $\totalization{f}$ defined as
 
\[\totalization{f}(x)\defas\begin{cases}
f(x) & \text{ if } x \in \dom(f),\\
Y & \text{otherwise.}
\end{cases}
\]
For more details on the totalization we refer the reader to \cite{CompletionOfChoice}.

To define the completion of a problem $f$ we need to first introduce the completion of a represented space. We adopt the following notation: given $p \in \Baire$ we define $\hat{p}_n$ to be $\str{}$ if $p(n)=0$, $\str{p(n)-1}$ otherwise; then $p-1$ is the concatenation of all the $\hat{p}_n$'s. For a represented space $\mathbf{X}=(X, \delta_X)$ we define its completion as $\overline{\mathbf{X}}=(\overline{X}, \delta_{\overline{X}})$ where $\overline{X}=X\cup \{\bot\}$ with $\bot \notin X$ and $\function{\delta_{\completion{X}}}{\Baire}{\completion{X}}$ is the total function defined by 
\[\delta_{\completion{X}}(p)\defas\begin{cases}
\delta_X(p-1)& \text{if } p-1 \in \dom(\delta_X)\\
\bot &\text{otherwise.}
\end{cases}\]

Let $\partialmultifunction{f}{\mathbf{X}}{\mathbf{Y}}$ be a multi-valued function. We define the completion of $f$ as the total multi-valued function $\multifunction{\completion{f}}{\completion{\mathbf{X}}}{\completion{\mathbf{Y}}}$ such that
\[\completion{f}(x)=
\begin{cases}
f(x) & \text{ if } x \in \dom(f),\\
\overline{Y} & \text{ otherwise}.
\end{cases}
\]

We now introduce some  major benchmarks in the Weihrauch lattice. 

The well-known problem $\function{\lpo}{\Cantor}{\{0,1\}}$ is defined as $\lpo(p)=1$ iff $(\forall n)(p(n)=0)$. It is convenient to think of $\lpo$ as the function answering yes or no to questions which are $\Pi_1^{0}$ or $\Sigma_1^{0}$ in the input. Similarly, $\lpo^{(n)}$ can be seen as the function answering yes or no to questions which are $\Pi_{n+1}^0$ or $\Sigma_{n+1}^0$ in the input. It is well-known that $\mflim \strongweiequiv \parallelization{\lpo}$. Moreover, using \cite[\S 5]{BolWei11}, we obtain that for every $n$, $\mflim^{(n)} \strongweiequiv \parallelization{\lpo^{(n)}}$.

We define also $\function{\wf}{\tree}{\{0,1\}}$ as $\wf(T)=1$ iff $T \in \wellfounded$. Analogously to $\lpo$, we can think of $\wf$ as the problem answering yes or not to questions which are $\Pi_1^{1}$ or $\Sigma_1^{1}$ in the input. In the literature, $\wf$ was introduced under different names: the same notation appears in \cite{leafmanaegement,valentisolda}, while \cite{CompletionOfChoice} uses $\mathsf{WFT}$ and \cite{kihara_marcone_pauly_2020,openRamsey} use $\chi_{\Pi_1^1}$.

We now move our attention to \emph{choice problems}, which have emerged as very significant milestones in the Weihrauch lattice. For a computable metric space $\X$ and a pointclass $\boldfaceGamma$ as the ones in \thref{Borelrepspaces}, let $\partialmultifunction{\codedChoice{\boldfaceGamma}{}{\X}}{\boldfaceGamma(\X)}{\X}$  be the problem that given in input a nonempty set $A \in \boldfaceGamma(\X)$ outputs a member of $A$. When $\boldfaceGamma=\boldfacePi_1^0$ we just write $\codedChoice{}{}{\X}$. The same problem with domain restricted to singletons is denoted by $\codedUChoice{\boldfaceGamma}{}{\X}$.  It is well-known that for every $n>0$, $(\codedChoice{\boldfacePi_n^0}{}{\nats})'\weiequiv \codedChoice{\boldfacePi_{n+1}^0}{}{\nats}$. 

 Using the tree representation of closed sets, $\CBaire$ can be formulated as the problem of computing a path through some $T \in \illfounded$; $\UCBaire$ is the same problem with domain restricted to $\uniquebranch$. Notice that both problems are closed under compositional product by \cite[Theorem 7.3]{closedChoice}.
As noticed in \cite{kihara_marcone_pauly_2020} and mentioned in the introduction, $\CBaire$ and $\UCBaire$ are among the problems that correspond to $\mathsf{ATR}_0$. We need the following proposition.

\begin{proposition}
\thlabel{UCbaireisparallelization}
    $\parallelization{\codedUChoice{\boldfacePi_1^1}{}{\nats}} \weiequiv \parallelization{ \codedChoice{\boldfacePi_1^1}{}{\nats}}\weiequiv\UCBaire\strictlyweireducible \parallelization{\codedChoice{\boldfaceSigma_1^1}{}{\nats}}\strictlyweireducible\CBaire$.
\end{proposition}
\begin{proof}
The reduction $\codedUChoice{\boldfacePi_1^1}{}{\nats} \weireducible \codedChoice{\boldfacePi_1^1}{}{\nats}$ is trivial; by the effective version of the Novikov-Kondo-Addison uniformization theorem \cite[Theorem 4E.4]{Moschovakis}, we obtain  $\codedChoice{\boldfacePi_1^1}{}{\nats}\weireducible \codedUChoice{\boldfacePi_1^1}{}{\nats}$. Hence, $\parallelization{\codedChoice{\boldfacePi_1^1}{}{\nats}} \weiequiv \parallelization{\codedUChoice{\boldfacePi_1^1}{}{\nats}}$.

By \cite[Theorem 3.11]{kihara_marcone_pauly_2020} $\UCBaire\weiequiv\parallelization{\codedUChoice{\boldfacePi_1^1}{}{2}}$ (in \cite{kihara_marcone_pauly_2020}  $\parallelization{\codedUChoice{\boldfacePi_1^1}{}{2}}$ is denoted by $\boldfaceDelta_1^1\text{-}\mathsf{CA}$). Since  $\parallelization{\codedUChoice{\boldfacePi_1^1}{}{2}}\weireducible\parallelization{\codedUChoice{\boldfacePi_1^1}{}{\nats}}$ is trivial, we obtain $\UCBaire \weireducible \parallelization{\codedUChoice{\boldfacePi_1^1}{}{\nats}}$.
To complete the proof of the equivalence between the first three problems, it suffices to show that $\codedUChoice{\boldfacePi_1^1}{}{\nats} \weireducible \parallelization{\codedUChoice{\boldfacePi_1^1}{}{2}}$ and then notice that $\parallelization{\codedUChoice{\boldfacePi_1^1}{}{2}}$ is parallelizable. We can think of an input for $\codedUChoice{\boldfacePi_1^1}{}{\nats}$ as a sequence  $(T^i)_{i \in \nats} \in \tree^\nats$ such that exactly one $T^i$ belongs to $\wellfounded$. For every $i$, we compute the pair of trees $(S^i,R^i)$ where $S^i\defas\underset{j\leq i}{*}T^j$ and $R^i\defas\underset{j> i}{*}T^j$: notice that exactly one of $S^i$ and $R^i \in \wellfounded$. We can view the sequence $(S^i,R^i)_{i \in \nats}$ as an instance of $\parallelization{\codedUChoice{\boldfacePi_1^1}{}{2}}$. Finally, let $n\defas\min \big\{m: \parallelization{\codedUChoice{\boldfacePi_1^1}{}{2}}((S^i,R^i)_{i \in \nats}))(m)=0 \big\}$.   Clearly $T^n \in \wellfounded$.

The last two (strict) reductions are \cite[Theorem 4.3]{kihara_marcone_pauly_2020} and  \cite[Theorem 3.34]{choiceprinciples}.
\end{proof}

For further reference, we collect here some facts which are implicit in the literature.

\begin{proposition}
\thlabel{prop:prefixes}
    Let $\mathbf{X}$ and $\mathbf{Y}$ be represented spaces. Let $\partialmultifunction{h}{\mathbf{X}}{\mathbf{Y}}$ and $\boldfaceGamma$ is an arithmetic or projective pointclass.     
    For every $x\in\dom(h)$, let $\mathbf{Prefixes}(x):=\{ \sigma\in\baire \st (\exists y\in h(x))(\exists q\in \delta_Y^{-1}(y))(\sigma \sqsubset q) \}$.    
    Assume that there exists a computable $\partialfunction{U}{\mathbf{X}}{\boldfaceGamma(\baire)}$ such that for every $x \in \dom(h)$, $U(x) \subseteq \mathbf{Prefixes}(x)$ and there exists $q \in \delta_Y^{-1}(h(x))$ such that $(\forall n \in \nats )(\exists \sigma \in U(x))(q[n] \sqsubseteq \sigma)$.
    Then $\firstOrderPart{h}\weireducible \codedChoice{\boldfaceGamma}{}{\nats}$.
\end{proposition}
\begin{proof}
    Let $\partialmultifunction{f}{\Baire}{\nats}$ be a partial multi-valued function and assume that $f\weireducible h$ via $\Phi,\Psi$. 
    For every $p\in\dom(f)$ and let $\Phi(p)$ be a name for $x\in\dom(h)$. Consider the set 
    \[A(x):=U(x) \cap \{\sigma \in \baire: \Psi(p[\length{\sigma}],\sigma)(0)\downarrow \}. \]
    Observe that $A(x) \neq \emptyset$: indeed, by hypothesis, there is $q\in\delta_Y^{-1}(h(x))$ such that $(\forall n \in \nats )(\exists \sigma \in U(x))(q[n] \sqsubseteq \sigma)$. Since $\Phi,\Psi$ witness a Weihrauch reduction, there is $n$ such that $\Psi(p[n],q[n])(0)\downarrow$. Let $\sigma \in U$ such that $q[n]\sqsubseteq \sigma$. By continuity, $\Psi(p[\length{\sigma}],\sigma)(0)\downarrow$ and therefore $\sigma\in A(x)$.
    
    Notice that
    $A(x) \in \Lambda^p$ where $\Lambda$ is:    
    \begin{itemize}
        \item $\Gamma$ if $\Gamma$ contains all $\Sigma_1^0$ sets;
        \item $\Sigma_2^0$  if $\Gamma=\Pi_1^0$;
        \item and $\Sigma_1^0$ if $\Gamma=\Delta_1^0$.
    \end{itemize}

It is easy to observe that $\codedChoice{\boldfaceSigma_{n+1}^0}{}{\nats} \weiequiv \codedChoice{\boldfacePi_n^0}{}{\nats}$ for every $n$, and hence, in any case, $\codedChoice{\boldfaceGamma}{}{\nats}  \weiequiv \codedChoice{\boldfaceLambda}{}{\nats}$. Applying $\codedChoice{\boldfaceLambda}{}{\nats}$ to $A(x)$ produces a string $\sigma$ such that $\Psi(p[\length{\sigma}],\sigma)(0)\in f(p)$ as $\sigma \in U(x) \subseteq \mathbf{Prefixes}(x)$.
\end{proof}

\begin{proposition}
\thlabel{Fopcbaire}
If  $\boldfaceGamma \in \{\boldfaceSigma,\boldfacePi,\boldfaceDelta\}$ and $\boldfaceLambda\in \{\boldfacePi,\boldfaceDelta\}$ then
 \[\firstOrderPart{\UCBaire}\weiequiv \codedUChoice{\boldfaceGamma_1^1}{}{\nats}\weiequiv\ustar{(\codedUChoice{\boldfaceGamma_1^1}{}{\nats})}\weiequiv \codedChoice{\boldfaceLambda_1^1}{}{\nats}\weiequiv\ustar{(\codedChoice{\boldfaceLambda_1^1}{}{\nats})}\strictlyweireducible\firstOrderPart{\CBaire}\weiequiv \codedChoice{\boldfaceSigma_1^1}{}{\nats}.\]
\end{proposition}
\begin{proof}
The equivalence $\firstOrderPart{\CBaire}\weiequiv \codedChoice{\boldfaceSigma_1^1}{}{\nats}$ is proved in \cite[Proposition 2.4]{goh_pauly_valenti_2021}; essentially the same proof shows that $\firstOrderPart{\UCBaire}\weiequiv \codedUChoice{\boldfaceSigma_1^1}{}{\nats}$. 

 If $A\subseteq \nats$ is a singleton then $n \in A$ iff $(\forall m\neq n)(m \notin A)$. This implies that $A$ is $\Sigma_1^1$ iff $A$ is $\Pi_1^1$ iff $A$ is $\Delta_1^1$ and this  shows that $\codedUChoice{\boldfaceSigma_1^1}{}{\nats}\weiequiv \codedUChoice{\boldfacePi_1^1}{}{\nats}\weiequiv\codedUChoice{\boldfaceDelta_1^1}{}{\nats}$. These equivalences, together with \thref{UCbaireisparallelization} and \thref{Summaryfopustar} allow us to derive all the stated equivalent characterizations of $\firstOrderPart{\UCBaire}$.
    
By \thref{UCbaireisparallelization}, $\UCBaire \strictlyweireducible \parallelization{\codedChoice{\boldfaceSigma_1^1}{}{\nats}}$: since $\UCBaire$ is parallelizable, this implies $\codedChoice{\boldfaceSigma_1^1}{}{\nats} \not\weireducible \UCBaire$ and hence $\firstOrderPart{\UCBaire} \strictlyweireducible \firstOrderPart{\CBaire}$.
\end{proof}

\begin{proposition}
\thlabel{limdoesnotreachucbaire}
For every $n$,
\begin{enumerate}[(i)]
    \item $\firstOrderPart{(\mflim^{(n)})} \weiequiv \codedChoice{\boldfacePi_{n+1}^0}{}{\nats} \weiequiv \ustar{(\lpo^{(n)})} \strictlyweireducible \firstOrderPart{(\mflim^{(n+1)})}$;
    \item $ \lpo^{(n+1)} \not\weireducible \codedChoice{\boldfacePi_{n+1}^0}{}{\nats}$;
    \item $\firstOrderPart{(\mflim^{(n)})} \strictlyweireducible   \codedChoice{\boldfacePi_1^1}{}{\nats} \weiequiv \firstOrderPart{\UCBaire}$. 
\end{enumerate}
\end{proposition}
\begin{proof}
The equivalences in (i) are from {\cite[Theorem 7.2]{valentisolda}}.
It follows from \cite[Theorem 5.10(5)]{valentisolda} that $\parallelization{\lpo^{(n)^{u*}}} \weiequiv \parallelization{\lpo^{(n)}} \weiequiv \mflim^{(n)}$; since  $\mflim^{(n)} \strictlystrongweireducible \mflim^{(n+1)}$ this implies the non-reductions in (i) and (ii).

The equivalence in (iii) is from \thref{Fopcbaire}, while the strictness follows from (i).
\end{proof}

\section{The perfect set theorem in $\Baire$ and $\Cantor$}
\label{perfectsetsgeneral}

\subsection{Perfect sets}
\label{perfectsets}
The following multi-valued function was introduced and studied in \cite{kihara_marcone_pauly_2020}.

\begin{definition}
The multi-valued function
$\partialmultifunction{\PTT[1]}{\tree}{\tree}$ has domain $\{T  \in \tree: T \in \uncountable\}$ and is defined by
$$\PTT[1](T)\defas\{S\in \tree: S\subseteq T \land S \text{ is perfect}\}. $$
\end{definition}
We also study $\PTT[1]{\restriction \tree_2}$, the restriction of $\PTT[1]$ to $\tree_2$. We now define the same problem for closed sets.
\begin{definition} 
Let $\X$ be a computable Polish space. The multi-valued function $\partialmultifunction{\PST[\X]}{\negrepr{\X}}{\negrepr{\X}}$ has domain $\{A  \in \negrepr{\X}: \length{A}>\aleph_0\}$ and is defined  as
$$\PST[\X](A)\defas\{P \in \negrepr{\X}: P\subseteq A \land P \text{ is perfect}\}.$$
\end{definition}

Using the tree representation of closed sets in $\Baire$, we can think of a name for an input of $\PST[\Baire]$ as  $T \in \uncountable$ and a name for a solution of $\PST[\Baire]([T])$ as  $S\in\tree$ such that $\body{S}\subseteq \body{T}$ and $\body{S}$ is perfect. Notice that $\PST[\Baire](\body{T})$ contains $\PTT[1](T)$, but includes also every tree with perfect body contained in $\body{T}$.
\begin{theorem}
\thlabel{cantor_and_baire_same}
$\PTT[1]{\restriction}\tree_2 \strongweiequiv \PTT[1]$ and $\PST[\Cantor]\strongweiequiv \PST[\Baire]$.
\end{theorem}
\begin{proof}
The reduction $\PTT[1]{\restriction}\tree_2\strongweireducible \PTT[1]$ is trivial.

We now prove that $\PST[\Cantor] \strongweireducible \PST[\Baire]$. Let $T \in \tree_2$ and let the forward functional be the identity. Let $P$ be a name for $\PST[\Baire]([T])$: even if $P\in\tree$, notice that $\body{P}$ is perfect and $\body{P}\subseteq \Cantor$. Let $\Psi(P)=\{\sigma \in P: \sigma \in \cantor\}$ and notice that $\Psi(P)$ is a name for an element $\PST[\Cantor](\body{T})$.  

For the other direction, the reduction $\PTT[1]\strongweireducible\PTT[1]{\restriction}\tree_2 $ is witnessed by the maps $\translateCantor$ (forward) and  $\translateBaire$ (backward) from \thref{Translatetrees}. Let $T\in \uncountable$ and let $P \in \PTT[1]\restriction \tree_2(\translateCantor(T))$. By \thref{Perfecttreesinbaire}, $\translateBaire(P)$ is a perfect tree. To show that $\translateBaire(P) \subseteq T$, it suffices to prove that $\body{\translateBaire(P)}\subseteq \body{T}$. Let $f \in \body{\translateBaire(P)}$: by \thref{AllPropertiesOfTranslation}(6) we have that $\translateCantor(f) \in \body{P}\subseteq \body{\translateCantor(T)}$ and by \thref{AllPropertiesOfTranslation}(4) we conclude that $f \in \body{T}$.

The proof that $\PST[\Baire] \strongweireducible \PST[\Cantor]$ is similar.
\end{proof}

\begin{lemma}
\thlabel{UCBaireReducesToPST}
$\UCBaire\strictlystrongweireducible\PST[\Baire]$ and $\PST[\Baire]\not\weireducible\UCBaire$.
\end{lemma}
\begin{proof}
Since  $\PST[\Baire]\strongweiequiv \PST[\Cantor]$ (\thref{cantor_and_baire_same}), we prove the lemma with $\PST[\Cantor]$ in place of $\PST[\Baire]$. 

To show $\UCBaire\strongweireducible\PST[\Cantor]$, fix $T \in \uniquebranch$ an let $p_0$ be the unique element of $\body{T}$. Let $S\defas \translateCantor(\exploded{T})$ and notice that, by definition of $\exploded{\cdot}$ and by \thref{AllPropertiesOfTranslation}(4), all paths in $ \body{S}$ are either eventually zero or are of the form $\translateCantor(p_0*q)$ for some $q \in \Cantor$.
 
 Fix a name $P$ for an element of $\PST[\Cantor](\body{S})$. We claim that all the paths in $\body{P}$ are of the form $\translateCantor(p_0*q)$ for some $q \in \Cantor$. To prove this, we need to rule out that some eventually zero path belongs to $\body{P}$. Let $r \in \Cantor$ be of the form $\sigma0^\nats$, where $\sigma=\str{}$ or  $\sigma(\length{\sigma}-1)=1$.  Notice that $\translateCantor(\translateBaire(\sigma))=\sigma$. Let 
 \[k\defas \begin{cases}
     p_0( {\length{\translateBaire(\sigma)}}/2) & \text{if } \length{\translateBaire(\sigma)} \text{ is even},\\
     1 & \text{otherwise,}
 \end{cases}\]
and set $m= \length{\sigma}+ k +1$. It suffices to prove that 
 $$(\forall q \in \Cantor)(r[m]\not\sqsubset \translateCantor(p_0*q)),$$
so that all paths in $\body{S}$ which extend $r[m]$ are eventually zero and $S_{r[m]} \in \countable$, which implies $r \notin \body{P}$.
Fix $q \in \Cantor$:
\begin{itemize}
    \item if $\sigma\not\sqsubset \translateCantor(p_0*q)$ then $r[m]\not\sqsubset \translateCantor(p_0*q)$;
\item if $\sigma\sqsubset \translateCantor(p_0*q)$ then either $\sigma 0^k 1$ or $\sigma 0^{k-1} 1$ (in case $\length{\translateBaire(\sigma)}$ is odd and $q((\length{\translateBaire(\sigma)} -1)/2) =0$) is a prefix of $\translateCantor(p_0*q)$ which is incomparable with $r[m] = \sigma 0^{k+1}$; hence $r[m] \not\sqsubset \translateCantor(p_0*q)$ also in this case.
\end{itemize}

We show how to computably retrieve $p_0$ from ${P}$. To find $p_0(0)$ we search for $n$ such that
 $$(\forall \tau \in 2^{n+1})( P_\tau \in \illfounded_2 \implies \tau = 0^{n}1 ).$$
Indeed, the previous claim implies that the unique $n$ satisfying this condition is $p_0(0)$. Since $\illfounded_2$  is a $\Pi_1^0$ set (\thref{Complexityresults}(i)), the above condition is $\Sigma_1^0$ and at some finite stage we find $p_0(0)$.

Suppose now that we computed the first $i$ coordinates of $p_0$, i.e. $p_0[i]$. We generalize the previous strategy to compute $p_0(i)$. Let  $$A_i\defas \{0^{p_0(0)} 1 \xi_0 0^{p_0(1)} 1 \xi_1 \dots \xi_{i-2} 0^{p_0(i-1)} 1 \xi_{i-1} \in P: (\forall j<i) (\xi_j \in \{1,01\})\}
$$
(recall that $\translateCantor(0)=1$ and $\translateCantor(1)=01$). Informally, the $\xi_j$'s come from the interleaving of sequences in $T$ with sequences in $\cantor$. Notice that $A_i$ is finite and nonempty; moreover, there exists $\sigma \in A_i$ which is the prefix of some path in $\body{P}$. We search for $n$ satisfying the $\Sigma_1^0$ property

  $$(\forall \sigma \in A_i)(\forall \tau \in 2^{n+1})(  P_{\sigma\tau} \in \illfounded_2 \implies  \tau = 0^{n}1  )$$
As before the claim implies that the unique $n$ satisfying this condition is $p_0(i)$. The main difference with the case $i=0$ is that it may be the case that different sequences in $A_i$ are prefixes of paths in $\body{P}$ that come from the interleaving of $p_0$ with different elements of $\Cantor$. However, any such $\sigma$ provides the correct $n = p_0(i)$.
This ends the proof of the reduction.\smallskip

 To show that $\PST[\Cantor]\not\strictlyweireducible\UCBaire$, recall from \S \ref{representedspaces} that $\mflim*\UCBaire\weiequiv \UCBaire$, so it  suffices to show that $\CBaire \weireducible \mflim*\PST[\Cantor]$. From \cite[Proposition 6.3]{kihara_marcone_pauly_2020} and \thref{cantor_and_baire_same}, it follows that  $\CBaire\weiequiv \PTT[1]\weiequiv \PTT[1]\restriction \tree_2$ and hence it is enough to show that $\PTT[1]\restriction \tree_2\weireducible \mflim*\PST[\Cantor]$. From \cite[Lemma 4.46]{Nobrega2017GameCA}, we know that $\mflim$ is equivalent to the function that prunes elements in $\tree_2$. So let $T\in \uncountable_2$ be the input of $\PTT[1]\restriction \tree_2$ and let $P$ be a name for an element of $\PST[\Cantor](\body{T})$:  pruning $P$ with $\mflim$ is enough to obtain a perfect subtree of $T$.
\end{proof}

 Recall that trees are represented via their characteristic functions and notice that $s \in \cantor$ is a prefix of a (name for a) tree iff $\{\tau: s(\tau)=1\}$ is a tree, which is a computable property.
\begin{lemma}
\thlabel{fop_pst}
$\firstOrderPart{\PST[\Baire]} \weiequiv \codedChoice{\boldfacePi^1_1}{}{\nats}$.
\end{lemma}
\begin{proof}
The fact that $\codedChoice{\boldfacePi_1^1}{}{\nats}\weireducible\firstOrderPart{\PST[\Baire]}$ follows from the fact that $\firstOrderPart{\UCBaire}\weiequiv\codedChoice{\boldfacePi^1_1}{}{\nats}$ by \thref{Fopcbaire} and $\UCBaire\strictlystrongweireducible\PST[\Baire]$ by \thref{UCBaireReducesToPST}.

For the opposite direction, with the goal of using \thref{prop:prefixes}, let $\partialfunction{U}{\tree}{\boldfacePi_1^1(\baire)}$ be defined by 
\[
U(T) :=\{s \in \baire: s \text{ is a prefix of a tree } \land (\forall \tau)(s(\tau)=0\implies T_\tau \in \countable)\}
\]
for any $T \in \uncountable$; $U(T)$ is $\boldfacePi_1^1$ as $\countable$ is $\Pi_1^1$ (see \thref{Complexityresults}(ii)).

The elements of $U(T)$ are the prefixes of the names for the perfect kernel of $T$, and hence belong to $\mathbf{Prefixes}(T)$.
Moreover if $q$ is any name for the perfect kernel of $T$ then $q[n] \in U(T)$ for all $n$.
Hence \thref{prop:prefixes} applies and we obtain $\firstOrderPart{\PST[\Baire]} \weireducible \codedChoice{\boldfacePi^1_1}{}{\nats}$.
\end{proof}
Our results about $\PST[\Cantor]$ and $\PST[\Baire]$ are summarized in the following theorem.
\begin{theorem}
\thlabel{pstucbairesummary}
$\UCBaire\strictlystrongweireducible \PST[\Cantor] \strongweiequiv \PST[\Baire] \strictlystrongweireducible \CBaire$
\end{theorem}
\begin{proof}
The first strict reduction and the first equivalence were proven in \thref{UCBaireReducesToPST} and \thref{cantor_and_baire_same} respectively. The last reduction follows from the fact $\PTT[1]\weiequiv \CBaire$ (\cite[Proposition 6.3]{kihara_marcone_pauly_2020}) and a solution for $\PTT[1]$ is also a solution for $\PST[\Baire]$. Strictness follows by \thref{fop_pst} as, by \thref{Fopcbaire}, $\codedChoice{\boldfacePi^1_1}{}{\nats}\strictlyweireducible\firstOrderPart{\CBaire}$.
\end{proof}
\subsection{Listing problems}
\label{listingproblems}
We now move our attention to the functions that, given in input a countable closed set of a computable metric space, output a list of all its elements. There are different possible meanings of the word \lq list\rq, and these correspond to different functions. For Baire and Cantor space, some of these functions were already introduced and studied in \cite{kihara_marcone_pauly_2020}. For trees and closed sets we made a distinction between the perfect tree and the perfect set theorem; on the other hand, if  $T \in \tree$ and  $A\in \negrepr{\Baire}$ are such that $A=\body{T}$ then listing the elements in $\body{T}$ and listing the elements in $A$ are the same problem.

We generalize \cite[Definition 6.1]{kihara_marcone_pauly_2020} from $\Baire$ to an arbitrary computable metric space.

\begin{definition}
\thlabel{definitionsList}
Let $\X$ be a computable metric space. The two multi-valued function $\partialmultifunction{\wList[\X]}{\negrepr{\X}}{(2\times \X)^\nats}$ and $\partialmultifunction{\List[\X]}{\negrepr{\X}}{\nats \times (2\times \X)^\nats}$ with the same domain $\{A \in \negrepr{\X}:\length{A}\leq \aleph_0\}$ are defined by
 $$\wList[\X](A)\defas\{(b_i,x_i)_{i \in \nats}: A=\{x_i:b_i=1\}\},$$
 $$\List[\X](A)\defas\{(n,(b_i,x_i)_{i \in \nats}): A=\{x_i:b_i=1\} \land ((n=0 \land \length{A}=\aleph_0) \lor (n>0 \land \length{A}=n-1)) \}.$$
\end{definition}


\begin{remark}
\thlabel{equivalentcersionslists}
Notice that there is a slight difference between $\List$ (there was no subscript there because only $\Baire$ was considered) in \cite[Definition 6.1]{kihara_marcone_pauly_2020} and our definition of $\List[\Baire]$. Indeed, $\partialmultifunction{\List}{\negrepr{\Baire}}{(\Baire)^\nats}$ is defined by stipulating that $(n,(p_i)_{i \in \nats}) \in \List(A)$ iff either $n=0$, $A=\{p_i :i \in \nats\}$ and $p_i \neq p_j$ for every $i \neq j$, or else $n>0$, $\length{A}=n-1$ and $A=\{p_i :i< n-1\}$.

In particular, the output of $\List$ is always injective: this version is apparently stronger than $\List[\Baire]$ because the latter allows repeating elements of the form  $(1,p_i)$ for $p_i \in A$. We briefly discuss why $\List \strongweiequiv \List[\Baire]$.

We claim that given $(n,(b_i,p_i)_{i \in \nats}) \in \List[\Baire](A)$ we can compute some $I$ such that $(n,I) \in \List[](A)$. At any finite stage $s$ we inspect the finite prefix of length $s$ of $(n,(b_i,p_i)_{i \in \nats})$, i.e., the prefix $n\rho$ such that $\rho$ is the dove-tailing of the strings $(b_i\sigma_i)_{i< m}$ for some $m \in \nats$. For any $i < m$ we list $p_i$ such that $\sigma_i \sqsubset p_i$ in $I$ if $b_i=1$ and $p_i \not\sqsubseteq p_j$ for any $j<i$ such that $\sigma_j \sqsubset p_j$ and $b_j=1$.


If $n>0$ (i.e.\ $A$ is finite) after we listed $n-1$ elements we can add to $I$ any element in $\Baire$. 
If $n=0$ (i.e.\ $A$ is infinite) then we always find new elements to list in $I$ and we continue forever.
Since we are listing each $p_i$ only if $p_i \neq p_j$ for every $j<i$ with $b_j=1$, we have that $I$ lists injectively all the elements of $A$.
\end{remark}

We now focus on listing problems in Cantor space, comparing $\wList[\Cantor]$ and $\List[\Cantor]$  with the analogous problems in Baire space and with the functions $\List[\Cantor,<\omega]$ and  $\wList[\Cantor, \leq \omega]$ considered in \cite[\S 6.1]{kihara_marcone_pauly_2020}. 
\begin{definition} The multi-valued functions $\partialmultifunction{\List[\Cantor,<\omega]}{\negrepr{\Cantor}}{(\Cantor)^{<\nats}}$ and $\partialmultifunction{\wList[\Cantor, \leq \omega]}{\negrepr{\Cantor}}{(\Cantor)^\nats}$ have domains $\{A \in \negrepr{\Cantor}:\length{A}<\aleph_0\}$ and $\{A \in \negrepr{\Cantor}:\length{A}\leq \aleph_0\land A\neq \emptyset\}$ respectively and are defined by
\begin{align*}
\List[\Cantor,<\omega](A) & \defas\{(p_i)_{i <n}: A=\{p_i:i<n\}\};\\    
\wList[\Cantor, \leq \omega](A) & \defas\{(p_i)_{i \in \nats}:A=\{p_i:i \in \nats\}\}.
\end{align*}
\end{definition}

The following theorem establishes the relations between the listing problems defined so far: the results stated in this theorem are collected with other ones in Figure \ref{Figureslist}.

\begin{theorem}
\thlabel{Summary_list_cantor}
$\List[\Cantor,< \omega] \weiincomparable  \wList[\Cantor, \leq \omega] \weiequiv\wList[\Cantor]$ and $\wList[\Cantor, \leq \omega],\List[\Cantor, < \omega]\strictlyweireducible \List[\Cantor] \strictlyweireducible \UCBaire \weiequiv \wList[\Baire] \weiequiv \List[\Baire]$.

Since all the problems involved are cylinders all the reductions mentioned above are strong. 
\end{theorem}

\begin{proof}

The fact that $\List[\Cantor,< \omega] \weiincomparable \wList[\Cantor, \leq \omega]$ is \cite[Corollary 6.15]{kihara_marcone_pauly_2020}.

To prove that $\wList[\Cantor, \leq \omega] \weireducible\wList[\Cantor]$, let $A \in \negrepr{\Cantor}$ be countable and nonempty and let $(b_i,p_i)_{i \in \nats} \in \wList[\Cantor](A)$. Then, $\{p_i:b_i=1\}$ can be easily rearranged, and possibly duplicated,  to produce an element of $\wList[\Cantor, \leq \omega](A)$.

For the opposite direction, let $A \in \negrepr{\Cantor}$ be countable and possibly empty. Define $A'\defas\{0^\nats\} \cup \{1 x: x \in A\}$: $A'$ is still countable but nonempty, i.e.\ a suitable input for $\wList[\Cantor, \leq \omega]$. Let $(p_i)_{i \in \nats}\in \wList[\Cantor, \leq \omega](A')$: then $(p_i(0),p_i^-)_{i \in \nats} \in \wList[\Cantor](A)$.

The reductions $\wList[\Cantor] \weireducible \List[\Cantor]$ and $\List[\Cantor,<\omega]\weireducible \List[\Cantor]$ are immediate. Strictness follows from incomparability of $\wList[\Cantor]$ and  $\List[\Cantor, < \omega]$.

By \thref{equivalentcersionslists} and \cite[Theorem 6.4]{kihara_marcone_pauly_2020} we obtain that $\UCBaire \weiequiv \wList[\Baire] \weiequiv \List[\Baire]$. 

The reduction $\List[\Cantor]\weireducible \List[\Baire]$ is obvious. 
To prove that $\List[\Baire]\not\weireducible \List[\Cantor]$ we recall that by \thref{limdoesnotreachucbaire}(i) and (iii) $\codedChoice{\boldfacePi_3^0}{}{\nats}\strictlyweireducible \codedChoice{\boldfacePi_1^1}{}{\nats}\weiequiv \firstOrderPart{\UCBaire}$. 
It thus suffices to show that $\firstOrderPart{\List[\Cantor]}\weireducible\codedChoice{\boldfacePi_3^0}{}{\nats}$ and to this end we use \thref{prop:prefixes}. 
Consider the formula 
\[
(\exists \sigma_0,\dots,\sigma_{n-1})(\forall i \neq j< n)(\sigma_i \incomparable \sigma_j \land T_{\sigma_i} \in \illfounded_2)
\]
that we denote by $\varphi(n,T)$. 
Since $\illfounded_2$ is $\Pi_1^0$ (see \thref{Complexityresults}(i)) $\varphi$ is $\Sigma_2^0$. 
Notice that in this case $\mathbf{Prefixes}(T)$ is the set of all $n\rho$ such that, if we read $\rho$ as the dove-tailing of the strings $(b_i\sigma_i)_{i< m}$ for some $m \in \nats$, then
\begin{itemize}
    \item $(\forall i < m)(b_i=1 \implies T_{\sigma_i} \in \illfounded_2)$;
    \item $(n=0 \land (\forall k)(\varphi(k,T))) \lor (n>0 \land \varphi(n-1,T) \land \lnot \varphi(n,T))$.
\end{itemize}
Since $\mathbf{Prefixes}(T)$ is $\Pi_3^{0,T}$ we can apply \thref{prop:prefixes} to $U:=T \mapsto \mathbf{Prefixes}(T)$, thereby concluding the proof.
\end{proof}
\begin{lemma}
\thlabel{wsclistcantorparallelizable}
$\wList[\Cantor]$ is parallelizable.
\end{lemma}
\begin{proof}
To prove the claim it suffices to show that $\parallelization{\wList[\Cantor]}\weireducible \wList[\Cantor]$. Given $(T^n)_{n \in \nats}$, an input for $\parallelization{\wList[\Cantor]}$, i.e.\ a sequence of binary trees with countable body,  compute $T\defas\binarydisjointunion{n \in \nats}{T^n}$ and notice that $T \in \tree_2$ (see \thref{Disjoint_union}). Given $(b_i,p_i) \in \wList[\Cantor](T)$, it is straightforward to check that $\{(b_i,p_i): b_i=1 \land 0^n1\sqsubset p_i\} \in \wList[\Cantor](T^n)$.
\end{proof}
We conclude this section by characterizing the first-order part of $\wList[\Cantor]$.
\begin{lemma}
\thlabel{fop_wlist}
$\firstOrderPart{\wList[\Cantor]}\weiequiv \codedChoice{}{}{\nats}$.
\end{lemma}
\begin{proof} 
We first show that $\firstOrderPart{\wList[\Cantor]}\weireducible \codedChoice{}{}{\nats}$, and to do so, we use \thref{prop:prefixes}.

Let $T\in \countable_2$ and notice that $\mathbf{Prefixes}(T)$ is the set of all $\rho \in \cantor$ such that if we read $\rho$ as the dove-tailing of the strings $(b_i\sigma_i)_{i< m}$ for some $m \in \nats$ then
$(\forall i < m)(b_i=1 \implies T_{\sigma_i} \in \illfounded_2)$.
Since $\illfounded_2$ is a $\Pi_1^{0}$ set (\thref{Complexityresults}(i)), $\mathbf{Prefixes}(T)$ is $\Pi_1^{0,T}$. Thus, we can apply \thref{prop:prefixes} to $U:=T \mapsto \mathbf{Prefixes}(T)$, thereby concluding the proof of the left-to-right direction.

We now show that $\codedChoice{}{}{\mathbb{N}}\weireducible \wList[\Cantor]$. Let $A \in \negrepr{\nats}$ be nonempty, and let $A^c[s]$ denote the enumeration of the complement of $A$ up to stage $s$. We compute the tree 
$$T\defas \{\str{}\} \cup \{0^n10^{s}: n \notin A^c[s]\}.$$
Notice that for every $n$, $n \in A$ iff  $0^n10^\nats \in \body{T}$. Given $(b_i,p_i)_{i \in \nats} \in \wList[\Cantor](\body{T})$, we computably search for some $i$ such that $b_i=1$ and $0^n1\sqsubset p_i$ for some $n \in \nats$ (such an $i$ exists because $A$ is nonempty). By construction, $n \in \codedChoice{}{}{\mathbb{N}}(A)$.
\end{proof}

\section{Functions arising from the Cantor-Bendixson Theorem}
\label{cantorbendixson}

\subsection{Perfect kernels}
We now move to the study of functions related to the perfect kernel. 

\begin{definition}
\thlabel{pkdefinition}
Let $\function{\PK}{\tree}{\tree}$ be the total single-valued function defined as $\PK(T)\defas S$ where $S$ is the perfect kernel of $T$. We denote by $\PK{\restriction}\tree_2$ the restriction of $\PK$ to $\tree_2$. 

Similarly, for a computable Polish space $\X$, let $\function{\PK[\X]}{\negrepr{\X}}{\negrepr{\X}}$ be the total single-valued function defined as  $\PK[\X](A)\defas P$ where $P$ is the perfect kernel of $A$.
\end{definition}

Notice that $\PK$ was already introduced by Hirst in \cite{leafmanaegement}, where he also proved the following theorem.

\begin{theorem}
\thlabel{Chipistrongpktree}
    $\parallelization{\wf}\strongweiequiv \PK$.
\end{theorem}

The following proposition summarizes some well-known facts about the relationship between $\CBaire$ and (the parallelization of) $\wf$.

\begin{proposition}
\thlabel{cbairepica}
   $\wf\not\weireducible \CBaire$ and  $\CBaire\strictlystrongweireducible \parallelization{\wf}$.
\end{proposition}
\begin{proof}
The fact that $\wf\not\weireducible \CBaire$ was already noticed in \cite[page 1033]{kihara_marcone_pauly_2020}. 

To show that $\CBaire\weireducible \parallelization{\wf}$ it suffices to notice that $\PTT[1]\weiequiv \CBaire$ (\cite[Proposition 6.3]{kihara_marcone_pauly_2020}) and $\PK$ clearly computes $\PTT[1]$. The strictness of the reduction is immediate by the first part of this proposition and the fact that $\CBaire$ is parallelizable.
\end{proof}

Let $\function{\mathsf{J}}{\Baire}{\Baire}$, $p \mapsto p'$ denote the Turing jump operator. As $\mflim \strongweiequiv \mathsf{J}$ (\cite[Lemma 8.9]{closedChoice}), $\parallelization{\wf}$ is strongly Weihrauch equivalent to the function computing the \emph{hyperjump} of a set. 
The hyperjump of $A \subseteq \nats$ can be defined, following \cite[Definition 4.12]{Rogers}, as
\[ \mathsf{HJ}(A)\defas\{z:\varphi_z^A\text{ is the characteristic function of a well-founded tree}\}.\]
The well-known fact that $\mathsf{HJ}(A)$ is a $\Pi_1^{1,A}$-complete subset of the natural numbers allows us to prove the following proposition.

\begin{proposition}
\thlabel{Wfcompositionalproduct}
$\mflim*\parallelization{\wf} \not\weireducible \parallelization{\wf}$ and hence  $\parallelization{\wf}$ is not closed under compositional product.
\end{proposition}
\begin{proof}
Towards a contradiction, suppose that $\mflim*\parallelization{\wf} \weireducible \parallelization{\wf}$. By the definition of compositional product (see \S \ref{representedspaces}) and the facts that $\mflim \strongweiequiv \mathsf{J}$ and $\mathsf{J} \circ \parallelization{\wf}$ is defined, let $\Phi$ and $\Psi$ witness  $\mathsf{J}\circ \parallelization{\wf} \weireducible \parallelization{\wf}$. 

Let $(T^i)_{i \in \nats}$ be a computable list of all computable elements $\tree$ and notice that $\parallelization{\wf}((T^i)_{i \in \nats}) \equiv_T \mathsf{HJ}(\emptyset)$. Then, $\Phi((T^i)_{i \in \nats})$ is a computable list of trees and $\parallelization{\wf}(\Phi((T^i)_{i \in \nats}))$ is Turing reducible to $\mathsf{HJ}(\emptyset) \equiv_T \mathsf{HJ}({\Phi((T^i)_{i \in \nats})})$. Therefore, $\Psi\big((T^i)_{i \in \nats}, \parallelization{\wf}(\Phi((T^i)_{i \in \nats}))\big) \leq_T \mathsf{HJ}(\emptyset)$ as well. On the other hand, $(\mathsf{J}\circ \parallelization{\wf}) ((T^i)_{i \in \nats})$ computes the Turing jump of $\mathsf{HJ}(\emptyset)$, which is not Turing reducible  $\mathsf{HJ}(\emptyset)$.
\end{proof}

As we did when dealing with $\PST[\Baire]$ and $\PST[\Cantor]$, to study $\PK[\Baire]$ and $\PK[\Cantor]$ we use the tree representation of $\negrepr{\Baire}$ and $\negrepr{\Cantor}$. The following is the analog of \thref{cantor_and_baire_same}.

\begin{proposition}
\thlabel{Cantor_and_baire_same2}
$\PK{\restriction}\tree_2\strongweiequiv\PK$ and $\PK[\Cantor]\strongweiequiv \PK[\Baire]$.
\end{proposition}
\begin{proof}
We follow the pattern of the proof of \thref{cantor_and_baire_same}. 

$\PK{\restriction}\tree_2\strongweireducible\PK$ is trivial and $\PK[\Cantor] \strongweireducible \PK[\Baire]$ is witnessed by the same functionals of the proof of $\PST[\Cantor]\strongweireducible\PST[\Baire]$ in \thref{cantor_and_baire_same}. Given $T \in \tree_2$, and a name $P$ for $\PK[\Baire](\body{T})$ we get $\body{\Psi(P)}=\body{P}$. Hence, $\Psi(P)$ is a name for $\PK[\Cantor](\body{T})$.

For the opposite directions, we only deal with $\PK \strongweireducible \PK{\restriction}\tree_2$, as the proof of $\PK[\Baire] \strongweireducible \PK[\Cantor]$ follows the same pattern. Again, the reduction is witnessed by the same functionals of the analogous proof in \thref{cantor_and_baire_same}. Fix $T\in \tree$ and let $P\defas \PK(\translateCantor(\body{T}))$. As before, set $\Psi(P)\defas\translateBaire(P)$. To prove the reduction, it suffices to show that $\length{\body{T}\setminus \body{\translateBaire(P)}}\leq \aleph_0$. We claim that $\body{T}\setminus \body{\translateBaire(P)}\subseteq \{\translateBaire(q): q \in \body{\translateCantor(T)}\setminus \body{P}\land (\exists^\infty i)(q(i)=1)\}$, which completes the proof as the set on the right-hand side is countable. If $p \in \body{T}\setminus \body{\translateBaire(P)}$, then by \thref{AllPropertiesOfTranslation}(4) and (6) we have that $\translateCantor(p) \in \body{\translateCantor(T)} \setminus \body{P}$. Moreover, $q\defas\translateCantor(p)$ has infinitely many ones and, by \thref{AllPropertiesOfTranslation}(2), $p=\translateBaire(q)$.
\end{proof}

\begin{proposition}
\thlabel{Pkbaire_parallelizable}
$\PK[\Cantor]$ and $\PK[\Baire]$ are (strongly) parallelizable.
\end{proposition}
\begin{proof}
To prove the statement, by \thref{Cantor_and_baire_same2}, it is enough to show that $\parallelization{\PK[\Baire]}\strongweireducible\PK[\Baire]$. Given $(T^i)_{i \in \nats}$ an input for $\parallelization{\PK[\Baire]}$, let  $P$ be a name for  $\PK[\Baire](\body{\disjointunion{i \in \nats}{T^i}})$. By the definitions of perfect kernel and disjoint union of trees we have that for every $i$, $\{\sigma : i\concat \sigma \in P\}$ is a name for $\PK[\Baire](\body{T^i})$.
\end{proof}

\begin{definition} 
Let $\sierpinski$ be the Sierpi\'nski space, which is the space $\{0_{\sierpinski},1_{\sierpinski}\}$ with the representation
\[\repmap{\sierpinski}(p)\defas
\begin{cases}
1_{\sierpinski} & \text{if } (\exists i)(p(i)\neq 0),\\
0_{\sierpinski} & \text{if } p=0^\nats.
\end{cases}\]

We define the multi-valued function $\function{\wfsierpinski}{\tree}{\sierpinski}$ as
\[\wfsierpinski(T)\defas
\begin{cases}
1_{\sierpinski} &\text{if } T \in \wellfounded,\\
0_{\sierpinski} &\text{if } T \in \illfounded.
\end{cases}
\]
\end{definition}

The next proposition shows that the main functions we consider in this section are cylinders, which implies that most  reductions we obtain in this section are strong.

\begin{proposition}
$\parallelization{\wfsierpinski}$, $\parallelization{\wf}$, $\PK$, $\PK{\restriction}\tree_2$, $\PK[\Cantor]$ and $\PK[\Baire]$ are cylinders.
\end{proposition}
\begin{proof}
All six functions are parallelizable (this is either obvious or follows from \thref{Pkbaire_parallelizable,Chipistrongpktree}) and hence it is enough to show that $\id$ strongly Weihrauch reduces to each of them. As $\parallelization{\wfsierpinski} \strongweireducible \parallelization{\wf} \strongweiequiv \PK \strongweiequiv \PK{\restriction}\tree_2$ (\thref{Cantor_and_baire_same2,Chipistrongpktree}) and $\PK[\Cantor] \strongweiequiv \PK[\Baire]$ (\thref{Cantor_and_baire_same2}), it suffices to show that $\id \strongweireducible \parallelization{\wfsierpinski}$ and $\id \strongweireducible \PK[\Cantor]$. 

For the first reduction let $p$ be an input for $\id$. For any $i,j \in \nats$ let 
\[
 T^{\pairing{i,j}}\defas\begin{cases}
\emptyset & \text{if } p(i)=j,\\
 \cantor & \text{if } p(i)\neq j.
 \end{cases}
\]
Let $\parallelization{\wfsierpinski}((T^{\pairing{i,j}})_{i,j \in \nats})=(a_{\pairing{i,j}})_{i,j \in \nats}$. To compute $p(i)$ we search for the unique $j$ such that $a_{\pairing{i,j}}=1_\sierpinski$ (recall that the set of names for $1_\sierpinski$ is $\boldfaceSigma_1^0$).
 
For the second reduction, recall that by \thref{UCBaireReducesToPST}, $\UCBaire\strongweireducible\PST[\Cantor]$, clearly $\PST[\Cantor]\strongweireducible\PK[\Cantor]$ and $\id \strongweireducible\UCBaire$.
\end{proof}

We now give a useful characterization of $\PK[\Baire]$.

\begin{theorem}
\thlabel{Pkcantor_idPiSigma}
$\PK[\Baire] \weiequiv \parallelization{\wfsierpinski}$.
\end{theorem}
\begin{proof}
Recall that by \thref{Cantor_and_baire_same2}, $\PK[\Baire]\weiequiv \PK[\Cantor]$ so that it suffices to show that $\PK[\Cantor] \weiequiv \parallelization{\wfsierpinski}$.

Let $T\in \tree_2$ be a name for an input of $\PK[\Cantor]$. Notice that $\{\sigma: T_\sigma \in \countable_2\}$ is $\Pi_1^{1,T}$ (see \thref{Complexityresults}(ii)) and hence, using \thref{Complexityresults}(i), we can compute from $T$ a sequence $(S(\sigma))_{\sigma \in \cantor} \in \tree^\nats$  such that $S(\sigma) \in \wellfounded$ iff $T_\sigma \in \countable_2$. Let $A\defas\{p \in \Cantor: (\forall n)(\wfsierpinski(S(p[n]))=0_{\sierpinski})\}$: since the set of names for $0_\sierpinski$ is $\boldfacePi_1^0$, $A \in \negrepr{\Cantor}$ and we can compute $U \in \tree_2$ with $\body{U}=A$. Notice that for any $\tau \in U$, $\tau$ is a prefix of a path through $\body{U}$ iff
$T_\tau \in \uncountable$. Therefore, $U$ is a name for $\PK[\Cantor](\body{T})$.

To show that $\parallelization{\wfsierpinski} \weireducible \PK[\Cantor]$, as $\PK[\Cantor]$ is parallelizable it suffices to prove that $\wfsierpinski \weireducible \PK[\Cantor]$. Let $T \in \tree$ be an  input for $\wfsierpinski$ and notice that $T \in \wellfounded$ iff $\exploded{T} \in \wellfounded$ iff $\translateCantor(\exploded{T}) \in \countable$. Hence, if $S$ is a name for $ \PK[\Cantor](\body{\translateCantor(\exploded{T})})$, $T \in \wellfounded$ iff $S \in \wellfounded_2$.
Since $\wellfounded_2$ is a $\Sigma_1^0$ set (see \thref{Complexityresults}(i)), given $S$ we can uniformly compute a name for $\wfsierpinski(T)$.
\end{proof}

\begin{proposition}
\thlabel{fop_pi11}
$\firstOrderPart{\PK[\Baire]} \weiequiv \codedChoice{\boldfacePi_1^1}{}{\nats}\strictlyweireducible \firstOrderPart{\CBaire}\weiequiv\codedChoice{\boldfaceSigma_1^1}{}{\nats}\strictlyweireducible\firstOrderPart{\parallelization{\wf}}\weiequiv\ustar{\wf}$.
\end{proposition}
\begin{proof}
Since clearly $\PST[\Baire]\weireducible\PK[\Baire]$,  by \thref{fop_pst} we obtain $\codedChoice{\boldfacePi_1^1}{}{\nats} \weiequiv \firstOrderPart{\PST[\Baire]} \weireducible \firstOrderPart{\PK[\Baire]}$. For the opposite direction, notice that the proof of $\firstOrderPart{\PST[\Baire]} \weireducible \codedChoice{\boldfacePi_1^1}{}{\nats}$ in \thref{fop_pst} actually shows that $\firstOrderPart{\PK[\Baire]} \weireducible \codedChoice{\boldfacePi_1^1}{}{\nats}$. 

\thref{Fopcbaire} tells us that $\codedChoice{\boldfacePi_1^1}{}{\nats}\strictlyweireducible\firstOrderPart{\CBaire}\weiequiv \codedChoice{\boldfaceSigma_1^1}{}{\nats}$. Recall that, by \thref{Summaryfopustar}, if $f$ is first-order then $\firstOrderPart{\parallelization{f}}\weiequiv \ustar{f}$. In particular, 
$\firstOrderPart{\parallelization{\wf}}\weiequiv \ustar{\wf}$. Besides, since $\CBaire$ is parallelizable, $\codedChoice{\boldfaceSigma_1^1}{}{\nats}\weiequiv \ustar{(\codedChoice{\boldfaceSigma_1^1}{}{\nats})}$. Since $\CBaire\strictlyweireducible \parallelization{\wf}$ (\thref{cbairepica}), we obtain
$ \codedChoice{\boldfaceSigma_1^1}{}{\nats}\strictlyweireducible\ustar{\wf}$.
\end{proof}

\begin{theorem}
\thlabel{Pk_below__chipi}
$\PK[\Baire]\strictlyweireducible\parallelization{\wf}\weiequiv \PK \weireducible \mflim * \PK[\Baire]$.
\end{theorem}
\begin{proof}
Given $T \in \tree$, $\PK(T)$ is a name for  $\PK[\Baire](\body{T})$: therefore $\PK[\Baire]\weireducible\PK$; strictness follows from  \thref{fop_pi11}. By \thref{Chipistrongpktree}, $\parallelization{\wf}\weiequiv \PK$.

To prove the last reduction, by \thref{Cantor_and_baire_same2} and \thref{Chipistrongpktree} we have that $\PK{\restriction}\tree_2\weiequiv \parallelization{\wf}$ and $\PK[\Baire]\weiequiv \PK[\Cantor]$, hence, to finish the proof, it suffices to show that $\PK{\restriction}\tree_2 \weireducible \mflim*\PK[\Cantor]$. From \cite[Lemma 4.46]{Nobrega2017GameCA}, we know that $\mflim$ is equivalent to the function that prunes a binary tree. So let $T\in \tree_2$ and let $P$ be a name for $\PK[\Cantor](\body{T})$:  pruning $P$ with $\mflim$ is enough to obtain $\PK{\restriction}\tree_2(T)$.
\end{proof}
We do not know whether $\PK \weiequiv \mflim * \PK[\Baire]$ (see \thref{PKandLimQuestion}).

\begin{proposition}
\thlabel{Ucbaire_below_pk}
 $\PST[\Baire]\strictlyweireducible \PK[\Baire]\weiincomparable \CBaire $. 
\end{proposition}
\begin{proof}
   The fact that $\PST[\Baire]\weireducible \PK[\Baire]$ is trivial. By \thref{Pk_below__chipi}, $\parallelization{\wf}\weireducible\mflim*\PK[\Baire]$ while, by the closure of $\CBaire$ under compositional product, we get $\mflim*\CBaire \weiequiv \CBaire$: hence by \thref{cbairepica} $ \PK[\Baire]\not\weireducible\CBaire$ and a fortiori $\PK[\Baire]\not\weireducible\PST[\Baire]$. For the opposite non-reduction, just notice that by \thref{fop_pi11}  we have that $\firstOrderPart{\CBaire} \not\weireducible \firstOrderPart{\PK[\Baire]}$. 
\end{proof}

\begin{proposition}
\thlabel{Pktcbaire}
$\PK[\Baire]\strictlyweireducible\parallelization{\totalization{\CBaire}}$ and hence the reduction $\wfsierpinski\weireducible\totalization{\CBaire}$ in \cite[Proposition 11.4(1)]{CompletionOfChoice} is actually strict.
\end{proposition}
\begin{proof}
 From $\wfsierpinski\weireducible\totalization{\CBaire}$ using \thref{Pkcantor_idPiSigma}, we obtain $\PK[\Baire] \weireducible \parallelization{\totalization{\CBaire}}$. Strictness follows from \thref{Ucbaire_below_pk} as $\CBaire \not\weireducible \PK[\Baire]$ but clearly $\CBaire \weireducible \totalization{\CBaire}$.
\end{proof}

We end the subsection by characterizing the deterministic part of $\PK[\Baire]$.

\begin{proposition}
\thlabel{Detpart_pkbaire}
$\deterministicPart(\PK[\Baire])\weiequiv \UCBaire$.
\end{proposition}
\begin{proof}
For the right to left direction, notice that $\UCBaire$ is single-valued and, by \thref{UCBaireReducesToPST} and \thref{Ucbaire_below_pk} we have that $\UCBaire \strictlyweireducible\PK[\Baire]$. For the converse, observe that $\firstOrderPart{\PK[\Baire]}\weiequiv \codedChoice{\boldfacePi_1^1}{}{\nats}$ (\thref{fop_pi11}) and $\parallelization{\codedChoice{\boldfacePi_1^1}{}{\nats}}\weiequiv\UCBaire$
(\thref{UCbaireisparallelization}). This, together with $\deterministicPart(\PK[\Baire])\weireducible \parallelization{\firstOrderPart{\PK[\Baire]}}$ (\cite[Corollary 3.7]{goh_pauly_valenti_2021}), concludes the proof.
\end{proof}

\subsection{Scattered lists}
\label{scatteredlist}

We now introduce the problems of listing the scattered part of a closed subset of a computable Polish space. Their definition is similar to \thref{definitionsList}: the crucial difference is that the domain includes all closed sets $A$ of the computable Polish space $\X$ (not only the countable ones as in \thref{definitionsList}) and we ask for the list of the elements of the scattered part of $A$.

\begin{definition}
Let $\X$ be a computable Polish space. We define three multi-valued functions $\multifunction{\wScList[\X]}{\negrepr{\X}}{2\times \X^\nats}$, $\function{\ScCount[\X]}{\negrepr{\X}}{\nats}$ and $\multifunction{\ScList[\X]}{\negrepr{\X}}{\nats \times (2\times \X)^\nats}$ by
\begin{align*}
\wScList[\X](A) & \defas \{(b_i,x_i)_{i \in \nats}: A\setminus\PK[\X](A) = \{x_i:b_i=1\}\},\\
\ScCount[\X](A) & \defas
 \begin{cases}
 0 & \text{if } A\setminus \PK[\X](A) \text{ is infinite},\\
 \length{A\setminus \PK[\X]}+1 & \text{if } A\setminus \PK[\X](A) \text{ is finite}.
 \end{cases}\\
\ScList[\X](A) & \defas \wScList[\X](A) \times \ScCount[\X](A).
\end{align*}
\end{definition}

\begin{remark}
\thlabel{Isolated_paths}
Notice that if a closed set $A$ of some $T_1$ topological space has a finite set $F$ of isolated points, then $A \setminus F$ is perfect and hence the scattered part of $A$ is $F$. Equivalently, if the scattered part of $A$ is infinite then it contains infinitely many isolated points. Moreover, the set of isolated points is always dense in the scattered part.
\end{remark}

With a similar proof to that \thref{Pkbaire_parallelizable}, we obtain the following.
\begin{proposition}
\thlabel{Wsclist_parallelizable}
$\wScList[\Baire]$ and $\wScList[\Cantor]$ are parallelizable.
\end{proposition}

One of the main results of this subsection is that $\wScList[\Baire]\weiequiv \PK[\Baire]$. We first prove the easier direction.
\begin{lemma}
\thlabel{Pkreduciblewsclist}
$\PK[\Baire]\weireducible\wScList[\Baire]$.
\end{lemma}
\begin{proof}
 By \thref{Pkcantor_idPiSigma} we get $\PK[\Baire] \weiequiv \parallelization{\wfsierpinski}$ and, by \thref{Wsclist_parallelizable}, $\wScList[\Baire]$ is parallelizable. So it suffices to show that $\wfsierpinski\weireducible \wScList[\Baire]$. 
       
Given an input $T \in \tree$ for $\wfsierpinski$, let $S\defas\binarydisjointunion{i \in \nats}{\exploded{T}}\in\tree$. If $T \in \wellfounded$ then $\exploded{T} \in \wellfounded$, so that by \thref{Disjoint_union}, $\body{S}=\{0^\nats\}$. If instead $T \in \illfounded$, $\body{\exploded{T}}$ is perfect and therefore $0^\nats \in \PK[\Baire](\body{S})$ and $\body{S}$ is perfect, so that the scattered part of $\body{S}$ is empty. 
      
Hence, for every $(b_i,x_i)_{i \in \nats} \in \wScList[\Baire](\body{S})$ we obtain $\wfsierpinski(T)=1_\sierpinski$ iff $(\exists i)( b_i=1)$.  Hence, a name for  $\wfsierpinski(T)$ can be uniformly computed from  $(b_i,x_i)_{i \in \nats}$.
\end{proof}

We split the proof of $\wScList[\Baire]\weireducible \PK[\Baire]$ in several lemmas. Before stating them we discuss some results in \cite[\S 6.1]{kihara_marcone_pauly_2020} and correct an error there.

\begin{remark}
\thlabel{Remarknnmcb}
Theorem 6.4 of \cite{kihara_marcone_pauly_2020} states (in our notation) that $\UCBaire\weiequiv \wList[\Baire] \weiequiv \List[\Baire]$. The main ingredient of the proof of $\wList[\Baire] \weireducible \UCBaire$ is a variant of the Cantor-Bendixson derivative that allows to carry out the process in a Borel way for countable closed sets. A single step of this process is called \emph{one-step $\mathsf{mCB}$-certificate} (\cite[Definition 6.5]{kihara_marcone_pauly_2020})  and all steps are then \lq\lq collected\rq\rq\ in a \emph{global $\mathsf{mCB}$-certificate} (\cite[Definition 6.6]{kihara_marcone_pauly_2020}). 
\end{remark}

\begin{definition}[{\cite[Definition 6.5]{kihara_marcone_pauly_2020}}]
Let $A \in \negrepr{\Baire}$. A \emph{one-step $\mathsf{mCB}$-certificate} for $A$ is some $c=((\sigma_i^c)_{i \in \nats},(b_i^c)_{i \in \nats},(p_i^c)_{i \in \nats}) \in (\baire,2,\Baire)^\nats$ where
\begin{itemize}
 \item  for all $i\neq j$, $\sigma_i^c \not\sqsubset \sigma_j^c$ and if $i<j$ then $\sigma_i^c<\sigma_j^c$;
 \item there exists $i$ such that $b_i^c=1$,
  \item Let $\mathsf{HYP}(A)$ be the set of hyperarithmetical elements of  $A$. For all $i$:
   \begin{itemize}
      \item if $b_i^c=1$, then $p_i^c \in A$ and $\sigma_i^c \sqsubset p_i^c$,
      \item if $b_i^c=0$, then $(\forall p \in \mathsf{HYP}(A))(\sigma_i^c \not\sqsubset p)$ and $p_i^c=0^\nats$,
    \item $(\forall p,q \in \mathsf{HYP}(A))(p,q \in A \land \sigma_i^c \sqsubset p,q \implies p=q)$,
\end{itemize}
   \item for every $\sigma \in \baire$, if $(\forall i \in \nats)(\sigma_i^c \not\sqsubseteq \sigma)$ then $(\exists p,q \in A)(p \neq q \land \sigma \sqsubset p,q)$. 
\end{itemize}
For a one-step $\mathsf{mCB}$-certificate $c$ for $A$, the residue of $c$ is $\{p \in A : (\forall i \in \nats) (\sigma_i^c \not\sqsubseteq p)\}$.
\end{definition}

 By \cite[Lemma 6.8]{kihara_marcone_pauly_2020}, every nonempty non-perfect $A \in \negrepr{\Baire}$ has a one-step $\mathsf{mCB}$-certificate $c$: moreover the residue of $c$ is the Cantor-Bendixson derivative of $A$. Furthermore, if $\length{A} \leq \aleph_0$ then the one-step $\mathsf{mCB}$-certificate of $A$ is unique.
 
 \begin{definition}[{\cite[Definition 6.6]{kihara_marcone_pauly_2020}}, corrected]
 \thlabel{Globalmcb}
A \emph{global $\mathsf{mCB}$-certificate} for $A \in \negrepr{\Baire}$ is indexed by some initial $I \subseteq \nats$ and consists of a sequence $(c_n)_{n \in I}$ and a strict linear ordering $\lhd$ on $I$ with minimum $n_0$ (if nonempty) such that, denoting with $A_i$ the residue of $c_i$:
\begin{itemize}
    \item $c_{n_0}$ is a one-step $\mathsf{mCB}$-certificate for $A$;
    \item for every $n \in I \setminus \{n_0\}$,  $c_n$ is a one-step $\mathsf{mCB}$-certificate for $\underset{i \lhd n}{\bigcap} A_i$;  
    \item for all $p \in \mathsf{HYP}(A)$, if $p \in A$ then $(\exists i \in I)(p \in A_i)$;
    \item for every $n,m \in I$, if $n<m$ then $\sigma_{h(n)}^{c_n}< \sigma_{h(m)}^{c_m}$ where $h(n)\defas \min\{i: b_{i}^{c_n}=1\}$.
\end{itemize}
\end{definition}

Observe that \thref{Globalmcb} differs from \cite[Definition 6.6]{kihara_marcone_pauly_2020} by the addition of the last requirement, which ensures that \cite[Corollary 6.9]{kihara_marcone_pauly_2020} holds: if $\length{A}\leq \aleph_0$, then $A$ has a unique global $\mathsf{mCB}$-certificate. Indeed, the last condition forces a specific ordering (determined by the code of the first finite sequence that is the prefix of an isolated path) on the sequence of the one-step $\mathsf{mCB}$-certificates, avoiding the different codings of the sequence allowed in \cite[Definition 6.6]{kihara_marcone_pauly_2020} because it was possible to permute the ordering.

Notice that, assuming $\length{A}\leq \aleph_0$, the global $\mathsf{mCB}$-certificate of $A$ computes a list of the paths in $A$. Since the global $\mathsf{mCB}$-certificate of a countable closed set $A$ is a $\Sigma_1^{1,A}$ singleton, we obtain, as in \cite[Theorem 6.4]{kihara_marcone_pauly_2020}, that $\wList[\Baire]\weireducible\UCBaire$. 

The following lemmas involve the completion of a problem (see \S \ref{representedspaces} for its definition).

\begin{lemma}
\thlabel{Lemmacompletion}
The function $\function{F}{\nats\times(\Baire\cup \baire)}{\completion{\Baire}}$ such that for all $e \in \nats$ and $p \in \Baire\cup \baire$
$$F(e,p)=\begin{cases}
\Phi_e(p)& \text{ if } p \in \Baire \text{ and } p \in \dom(\Phi_e),\\
\bot & \text{ otherwise,}
\end{cases}$$
is computable.
\end{lemma}
\begin{proof}
A computable realizer $\Phi'$ for $F$ can be defined recursively as follows. Suppose we have defined $\Phi'(e,p)[s]$ and let $t_s=\length{\{t<s:\Phi_e(p)(t)>0\}}$. Then set
\[
\Phi'(e,p)(s)=
\begin{cases}
\Phi_e(p)(t_{s})+1 & \text{ if } \Phi_e(p)(t_{s})\downarrow \text{in less than } s \text{ steps},\\
0& \text{ otherwise. }
\end{cases}\]
It is easy to check that this works.
\end{proof}

\begin{lemma}
\thlabel{completionofwlist_below_pk}
$\completion{\wList[\Baire]}\weireducible \PK[\Baire]$.
\end{lemma}
\begin{proof}
Since $\PK[\Baire]\weiequiv \PK[\Cantor]$ is parallelizable (\thref{Cantor_and_baire_same2,Pkbaire_parallelizable}), it suffices to show that $\completion{\wList[\Baire]}\weireducible \parallelization{\PK[\Cantor]}$.

By \cite[Lemma 5.1]{CompletionOfChoice}, $\negrepr{\Baire}$ is multi-retraceable, i.e.\ there is a computable multi-valued function $\multifunction{r}{\completion{\negrepr{\Baire}}}{{\negrepr{\Baire}}}$ such that its restriction to $\negrepr{\Baire}$ is the identity. Given $A\in \completion{\negrepr{\Baire}}$, let $G$ be the set of global $\mathsf{mCB}$-certificates of $r(A)\in \negrepr{\Baire}$. By \cite[Lemma 6.7]{kihara_marcone_pauly_2020},  $G$ is a $\Sigma_1^{1,A}$ subset of $\Baire$ and, in case $\length{r(A)}\leq \aleph_0$, by  \cite[Corollary 6.9]{kihara_marcone_pauly_2020}, $G$ is a singleton. 

Let $T$ be a name for $G$ as described in \S \ref{Background}, i.e.\  $T\in \tree$ is such that $G=\{x:(\exists y)(\forall n)(x[n]*y[n]\in T)\}$. Then, for every $m$, compute the tree
  $$T^{*m}:=\nats^m \cup \{\sigma:\str{\sigma(m),\dots,\sigma({\length{\sigma}-1})}\in T \land (\forall i<m)(m+2i<\length{\sigma}\implies \sigma(i)=\sigma(m+2i))\}.$$
Notice that $\body{T^{*m}}=\{\str{p(0),p(2),\dots,p(2m-2)}p:p \in \body{T}\}$;  therefore, every path in $\body{T^{*m}}$ begins with the first $m$ coordinates of some element of the analytic set $G$. In particular, if $G=\{p_0\}$ then every path in $\body{T^{*m}}$ extends $p_0[m]$. Let $U^m$ be a name for $\PK[\Cantor](\translateCantor(\exploded{T^{*m}}))$. If $G=\{p_0\}$ then every path in $\body{U^m}$ extends $\translateCantor(p_0[m]*\sigma)$ for some $\sigma \in 2^m$.
We now describe how to compute an element $x \in \Baire \cup \baire$ from the sequence $(U^m)_{m \in \nats}$ such that $x=p_0$ when $G=\{p_0\}$.

The procedure is similar to the one used in the proof of \thref{UCBaireReducesToPST}. Looking first at $U^1$, we search for $n_0$ such that
\begin{equation*}
    (\forall \tau \in 2^{n_0+1})(U_\tau^1 \in \illfounded_2 \implies \tau=0^{n_0}1).
\end{equation*}
Since $\illfounded_2$ is $\Sigma_1^0$ (\thref{Complexityresults}(i)), the above condition is $\Pi_1^{0}$. If we find such an $n_0$, we set $x(0)=n_0$ and we move to the next step. Notice that, in case $G=\{p_0\}$, the unique $n_0$ satisfying the above condition is $p_0(0)$.

Suppose we have computed $x[m-1]\defas n_0n_1\dots n_{m-1}$. We generalize the previous strategy to compute $n_m$. Let $$A_m=\{0^{n_0} 1 \xi_0 0^{n_1} 1 \xi_1 \dots \xi_{m-2} 0^{n_{m-1}} 1 \xi_{m-1} \in U^m: (\forall j<m) (\xi_j \in \{1,01\})\}.$$
We search for $n_m$ satisfying the $\Sigma_1^{0}$ property
  $$(\forall \sigma \in A_m)(\forall \tau \in 2^{n_m+1})(U_{\sigma\tau}^m \in \illfounded_2 \implies \tau=0^{n_m}1   )$$
As before, if we find such an $n_m$ we let $x(m)\defas n_m$ and move to the next step. Again, if $G=\{p_0\}$, the unique $n_m$ satisfying the above condition is $p_0(m)$.

The proof of \cite[Theorem 6.4]{kihara_marcone_pauly_2020} gives us a computable function $\Phi_e$ such that, in case $\length{r(A)}\leq \aleph_0$, $\Phi_e(x)$ is a name for a member of $\wList[\Baire](r(A))$. If $F$ is the function of \thref{Lemmacompletion} then, identifying the completion of the codomain of $\wList[\Baire]$ with $\completion{\Baire}$, we obtain $F(e,x) \in \completion{\wList[\Baire]}(r(A))$.

Summing up, if $A \in \negrepr{\Baire}$ is countable then $F(e,x)$ is a name in $\completion{\Baire}$ for an element of $\wList[\Baire](A)$. If instead $A \in \completion{\negrepr{\Baire}}$ does not belong to $\dom(\wList[\Baire])$, then $F(e,x)$ is anyway a name for some member of $\completion{\Baire}$.
\end{proof}

\begin{lemma}
\thlabel{wsclistreduciblepkbaire} 
$\wScList[\Baire] \weireducible \parallelization{\wfsierpinski} \times \parallelization{\completion{\wList[\Baire]}} .$
\end{lemma}
\begin{proof}

Let $T \in \tree$ be a name for an element of $\negrepr{\Baire}$ and recall that, by \thref{Complexityresults}(ii), $\{\sigma: T_\sigma \in \countable\}$ is $\Pi_1^{1,T}$. Hence, using \thref{Complexityresults}(i), we can compute $(S(\sigma))_{\sigma \in \baire} \in \tree^\nats$ such that $ S(\sigma) \in \wellfounded$ iff $T_\sigma \in \countable$. For any $\sigma \in T$, let $S(\sigma)$ and $\body{T_{\sigma}}$  be the inputs for the $\sigma$-th instance of $\wfsierpinski$ and  $\completion{\wList[\Baire]}$ respectively. Let $L_\sigma \in \completion{\wList[\Baire]}(\body{T_\sigma})$.

For any $\sigma$, when we see that $\wfsierpinski(S(\sigma))=1_{\sierpinski}$ then we know that $S(\sigma) \in \wellfounded$ and hence $T_\sigma \in \countable$: we need to include a list of $\body{T_\sigma}$ in $\wScList[\Baire](\body{T})$. We  compute a name for $L \in \wScList[\Baire](\body{T})$ combining all $L_\sigma$ such that $\wfsierpinski(S(\sigma))=1_{\sierpinski}$.
\end{proof}

\begin{theorem}
\thlabel{Pkbaire_equiv_wsclistbaire}
$\wScList[\Baire] \weiequiv \PK[\Baire]$.
\end{theorem}
\begin{proof}
The right-to-left direction is \thref{Pkreduciblewsclist}. For the opposite direction,
notice that, by \thref{Pkcantor_idPiSigma} and \thref{completionofwlist_below_pk}, we have that $\wfsierpinski,\completion{\wList[\Baire]}\weireducible \PK[\Baire]$, By \thref{wsclistreduciblepkbaire} and the fact that  $\PK[\Baire]$ is parallelizable  (\thref{Pkbaire_parallelizable}), we conclude that $\wScList[\Baire] \weireducible \PK[\Baire]$.
\end{proof}

We now study $\ScList[\Baire]$ and show that it lies strictly between $\PK[\Baire]$ and $\parallelization{\wf}$.
 
\begin{lemma}
\thlabel{wfustarsccountbaire}
$\wf^*\weireducible \ScCount[\Baire]$.
\end{lemma}
\begin{proof}
Let $(T^m)_{m\leq n} \in \tree^n$ be an input for $\wf^*$. For every $T^m$ compute the tree $S^m \defas \{j^n\tau:j<2^m \land n \in \nats \land \tau \in \exploded{T^m}\} $.
Notice that if $T^m\in \wellfounded$ then $\exploded{T^m} \in \wellfounded$, and, in this case, $\length{\body{S^m}}=2^m$. On the other hand, if $T^m \in \illfounded$ then $\body{S^m}$ is perfect. Now let $k \defas \ScCount[\Baire]\left(\body{\disjointunion{m\leq n}{S^m}}\right)$. Notice that $k>0$ (as the scattered part of $\body{\disjointunion{m\leq n}{S^m}}$ is always finite) and
\[
k-1 = \sum_{T^m \in \wellfounded} 2^m.
\]
Hence, the binary expansion of $k-1$ contains the information about which $T^m$'s are well-founded.
\end{proof}

We do not know more about the Weihrauch degree of $\ScCount[\Baire]$ (see \thref{question:sccount}).

\begin{theorem}
\thlabel{Pkbaire_below_sclistbaire}
$\PK[\Baire]\strictlyweireducible\ScList[\Baire]$.
\end{theorem}
\begin{proof}
Since $\wScList[\Baire]\weireducible\ScList[\Baire]$ is trivial, \thref{Pkbaire_equiv_wsclistbaire} immediately implies the reduction.

For strictness, first notice that $\wf \not\weireducible \firstOrderPart{\PK[\Baire]}$: indeed, by \thref{cbairepica} and the fact that $\wf$ is first-order, we obtain that  $\wf \not\weireducible \firstOrderPart{\CBaire}$ and by \thref{fop_pi11}, we get that $\firstOrderPart{\PK[\Baire]}\weireducible \firstOrderPart{\CBaire}$. Hence,$\wf\not\weireducible \PK[\Baire]$. On the other hand, clearly $\ScCount[\Baire] \weireducible \ScList[\Baire]$ and  $\wf^* \weireducible \ScCount[\Baire]$ by \thref{wfustarsccountbaire}, so that $\wf\weireducible \ScList[\Baire]$. 
\end{proof}

Combining this with the fact that $\PK[\Baire]\weiincomparable \CBaire$ (\thref{Ucbaire_below_pk}), we immediately obtain that $\ScList[\Baire]\not\weireducible \CBaire$. On the other hand, we do not know if $\CBaire \strictlyweireducible \ScList[\Baire]$ or $\CBaire \weiincomparable \ScList[\Baire]$ (see \thref{questioncbaire}).

\begin{remark}
\thlabel{Pi11sets2}
By \thref{Complexityresults}(iii), given a tree $T$, $\{\sigma : \length{\body{T_\sigma}}=1\}$ is $\Pi_1^{1,T}$ if $T \in \tree$ and $\Pi_2^{0,T}$ if $T \in \tree_2$.  Let $\varphi(n,T)\defas(\exists \sigma_0,\dots,\sigma_{n-1})(\forall i \neq j< n)(\sigma_i \incomparable \sigma_j \land \length{\body{T_{\sigma_i}}}=1)$: this formula, asserting that $T$ has at least $n$ isolated points, is $\Pi_1^{1}$ if $T \in \tree$, and $\Sigma_3^0$ if $T \in \tree_2$. Notice that, by \thref{Isolated_paths}, the scattered part of $\body{T}$ has at least $n$ elements iff $\varphi(n,T)$ holds. Therefore, the scattered part of $\body{T}$ is infinite iff $(\forall n)(\varphi(n,T))$, which is $\Pi_1^{1}$ if $T \in \tree$ and $\Pi_4^0$ if $T \in \tree_2$.
\end{remark}

\begin{theorem}
\thlabel{sclist_below_pica}
$\ScList[\Baire] \strictlyweireducible \parallelization{{\wf}}$.
\end{theorem}
\begin{proof}
To prove the reduction notice that $\ScList[\Baire] \weireducible \ScCount[\Baire] \times \wScList[\Baire]$. By \thref{Pkbaire_equiv_wsclistbaire} and  \thref{Pk_below__chipi}, $\wScList[\Baire]\weiequiv\PK[\Baire]\strictlyweireducible\widehat{\wf}$. As $\widehat{\wf}$ is  clearly parallelizable, it suffices to prove that $\ScCount[\Baire]\weireducible\widehat{\wf}$. Let $T \in \tree$ be a name for an input of $\ScCount[\Baire]$: by \thref{Pi11sets2} we can compute $(S^n)_{n \in \nats} \in \tree^\nats$ such that $S^0 \in \wellfounded$ iff the scattered part of $\body{T}$ has infinitely many elements and,for $n>0$, $S^n \in \wellfounded$ iff the scattered part of $\body{T}$ has at least $n$ elements. Then,
$$\ScCount[\Baire](\body{T})=\begin{cases}
0 & \text{if } \wf(S^0)=1,\\
\min\{n>0:\wf(S^{n})=0\} & \text{if } \wf(S^0)=0.
\end{cases}$$

For strictness, notice that $\ScList[\Baire]\weireducible \wScList[\Baire]*\ScCount[\Baire]$. By Theorem 3.9 of \cite{goh_pauly_valenti_2021}, $\deterministicPart(f*g)\weireducible\deterministicPart(f)*g$ and so 
$$\deterministicPart(\wScList[\Baire]*\ScCount[\Baire])\weireducible\deterministicPart(\wScList[\Baire])*\ScCount[\Baire]\weiequiv \UCBaire*\ScCount[\Baire],$$
where the equivalence follows from \thref{Pkbaire_equiv_wsclistbaire} and \thref{Detpart_pkbaire}. Since the output of $\ScCount[\Baire]$ is a natural number and the solution of $\UCBaire$ is always hyperarithmetical relative to the input (\cite[Corollary 3.4]{kihara_marcone_pauly_2020}), while $\parallelization{\wf}$ has instances with no hyperarithmetical solutions in the input, we conclude that $\deterministicPart(\parallelization{\wf})\weiequiv\parallelization{\wf}\not\weireducible\UCBaire*\ScCount[\Baire] $ (the equivalence is immediate as $\parallelization{\wf}$ is single-valued). Therefore, $ \parallelization{\wf} \not\weireducible\wScList[\Baire]*\ScCount[\Baire]$ and, a fortiori, $\parallelization{\wf} \not \weireducible \ScList[\Baire]$.
\end{proof}

We now move our attention to listing problems of the scattered part of closed subsets of Cantor space.  From \thref{Pk_below__chipi} and \thref{Pkbaire_equiv_wsclistbaire} notice that $\wf\weireducible\lpo*\wScList[\Baire]$. As the next lemma shows, to compute $\wf$ it suffices to compose $\wScList[\Cantor]$ with a function slightly stronger than $\lpo$.

\begin{lemma}
    \thlabel{Wsclistreacheswf}
    $\wf \weireducible \lpo'*\wScList[\Cantor]$
\end{lemma}
\begin{proof}
   Given an input $T\in \tree$ for $\wf$, we can compute $S\defas\binarydisjointunion{n \in \nats}{(\translateCantor(\exploded{T}))}$. Let $(b_i,p_i)_{i \in \nats} \in \wScList[\Cantor](S)$. Then
\begin{align*}
    T \in \wellfounded  & \iff \exploded{T}\in \wellfounded\\
    & \iff \translateCantor(\exploded{T}) \in \countable_2\\
    & \iff (\exists i)(b_i=1 \land p_i=0^\nats).
\end{align*}  
The last condition is $\Sigma_2^{0}$ and so $\lpo'$ suffices to establish from $(b_i,p_i)_{i \in \nats}$ whether $T \in \wellfounded$.
\end{proof}

\begin{lemma}
\thlabel{Fop_wsclistcantor}
$\codedChoice{\boldfacePi_2^0}{}{\nats}\weiequiv \firstOrderPart{\wScList[\Cantor]}$.
\end{lemma}
\begin{proof}

For the left-to-right direction, observe that  $\wScList[\Cantor]$ is parallelizable (\thref{Wsclist_parallelizable}), hence we show that $\codedChoice{\boldfacePi_2^0}{}{\nats} \weireducible\parallelization{\wScList[\Cantor]}$. An input for $\codedChoice{\boldfacePi_2^0}{}{\nats}$ is a nonempty set $A \in \boldfacePi_2^0(\nats)$. We can uniformly find a sequence $(p_n)_{n \in \nats}$ of elements of $\Cantor$  such that $n \in A \iff (\exists^{\infty} i)(p_n(i)=0)$.

For every $n$, let 
\[
T^n\defas \{\sigma \in \cantor: (\forall i <\length{\sigma})(p_n(i)=0\implies (\forall j<i)( \sigma(j)=0 )\}.
\]
Notice that if $(\exists^{\infty} i)(p_n(i)=0)$ then $\body{T^n}=\{0^\nats\}$, while $\body{T^n}$ is perfect otherwise.

Given $((b_{i,n},p_{i,n})_{i \in \nats})_{n \in \nats}\in \parallelization{\wScList[\Cantor]}((T^n)_{n \in \nats})$ notice that, for every $n$, $n \in A$ iff there exists $i$ such that $b_{i,n}=1$. Hence, we can find $n \in A$ simply by searching for a pair $i,n$ such that $b_{i,n}=1$.

For the other direction we use \thref{prop:prefixes}. Given given $T \in \tree_2$ let $U(T)$ be the set of $\rho \in \baire$ that, when read as the dove-tailing of the strings $(b_i\sigma_i)_{i< m}$ for some $m \in \nats$ satisfy $(\forall i < m)(b_i=1 \implies T_{\sigma_i} \in \uniquebranch_2)$.
Then $U(T) \subseteq \mathbf{Prefixes}(T)$ and $U(T)$ is uniformly $\Pi_2^{0,T}$ because $\uniquebranch_2$ is a $\Pi_2^0$ set (\thref{Complexityresults}(iii)). 

We now check that $U(T)$ satisfy the last hypothesis of  \thref{prop:prefixes}.
For any $q\in\wScList[\Cantor](\body{T})$ and any $n$ we read $q[n]$ as the dove-tailing of $(b_i\sigma_i)_{i< m}$ for some $m \in \nats$. 
If $b_i=1$ then $\sigma_i$ is a prefix of a member of the scattered part of $\body{T}$ and, by \thref{Isolated_paths}, there exists $\xi_i \sqsupseteq \sigma_i$ such that $T_{\xi_i} \in \uniquebranch_2$. 
Let $\rho \sqsupseteq q[n]$ be the dove-tailing of $(b'_i\sigma'_i)_{i<m'}$ (for some $m' \geq m$) such that $(\forall i<m)(b_i=1 \implies \sigma'_i \sqsupseteq \xi_i)$ and $(\forall i<m')(i \geq m \implies b'_i=0)$. Clearly $\rho \in U(T)$ and this concludes the proof.
\end{proof}



We collect in the next theorem our results about the problems of listing the scattered part of a closed set; these results are also summarized in Figure \ref{Figureslist}.

\begin{theorem}
\thlabel{sclistcantor_summary}
The following relations hold:
\begin{enumerate}[(i)]
    \item $ \List[\Cantor,< \omega],\wList[\Cantor] \strictlyweireducible \wScList[\Cantor]$, while $\List[\Cantor]$ and $\UCBaire$ are both Weihrauch incomparable with $\wScList[\Cantor]$;
    \item $\List[\Cantor],\wScList[\Cantor]\strictlyweireducible\ScList[\Cantor]$ and $\UCBaire\weiincomparable\ScList[\Cantor]$;
    \item  $\UCBaire,\ScList[\Cantor]\strictlyweireducible\wScList[\Baire]\strictlyweireducible\ScList[\Baire]$ while $\CBaire\weiincomparable \wScList[\Baire]$ and $\ScList[\Baire]\not\weireducible \CBaire$.
\end{enumerate}
\end{theorem}
\begin{proof}
\begin{enumerate}[(i)]
    \item By \cite[Proposition 6.11]{kihara_marcone_pauly_2020} and \thref{Fop_wsclistcantor}, $\List[\Cantor,< \omega]\weiequiv\codedChoice{\boldfacePi_2^0}{}{\nats}\weiequiv \firstOrderPart{\wScList[\Cantor]}$. The reduction $\wList[\Cantor]\weireducible\wScList[\Cantor]$ is obvious. From \thref{Summary_list_cantor} we know that $\List[\Cantor,< \omega] \weiincomparable \wList[\Cantor]$ and hence $\wScList[\Cantor] \not\weireducible\List[\Cantor,< \omega]$ and $\wScList[\Cantor] \not\weireducible \wList[\Cantor]$. 

   For the incomparabilities, recall that, by \thref{Summary_list_cantor}, $\List[\Cantor]\weireducible\UCBaire$: therefore it suffices to show that $\wScList[\Cantor]\not\weireducible \UCBaire$ and $\List[\Cantor]\not\weireducible\wScList[\Cantor]$.
   
   By \thref{Wsclistreacheswf}, $\wf\weireducible\lpo'* \wScList[\Cantor]$ while, since $\wf\not\weireducible\UCBaire$ (\thref{cbairepica}) and $\UCBaire$ is closed under compositional product, we obtain that $\wf\not\weireducible\lpo'*\UCBaire$. Hence, $\wScList[\Cantor]\not\weireducible \UCBaire$.
   
   To show that $\List[\Cantor]\not\weireducible \wScList[\Cantor]$ we prove that $\firstOrderPart{\List[\Cantor]}\not\weireducible \firstOrderPart{\wScList[\Cantor]}$. Since, by \thref{limdoesnotreachucbaire}(ii) and \thref{Fop_wsclistcantor}, $\lpo''\not\weireducible \codedChoice{\boldfacePi_2^0}{}{\nats}\weiequiv\firstOrderPart{\wScList[\Cantor]}$, it suffices to show $\lpo'' \weireducible \List[\Cantor]$. We can view an input for $\lpo''$ as a sequence $(q_n)_{n \in \nats}$ of elements of $\Cantor$ so that
   $$\lpo''((q_n)_{n \in \nats})=1 \iff (\exists^\infty n)(\forall i)(q_n(i)=0).$$
   Given $(q_n)_{n \in \nats}$, we compute $(T^n)_{n \in \nats} \in \tree^\nats$ defined as $T^n\defas \{0^s:(\forall i<s)(q_n(i)=0)\}$. Notice that $T^n \in \illfounded_2 \iff q_n=0^\nats$ and given $T'\defas \binarydisjointunion{n \in \nats}{T^n}$ it is easy to check that $\length{\body{T'}}=\aleph_0 \iff \lpo''((q_n)_{n \in \nats})=1$. Since the information about the cardinality of $\body{T'}$ is included in $\List[\Cantor](\body{T'})$, this concludes the reduction $\lpo''\weireducible \List[\Cantor]$.
 
 \item  The reductions $\List[\Cantor],\wScList[\Cantor]\weireducible\ScList[\Cantor]$ are immediate and, since we just showed that  $\List[\Cantor]\weiincomparable\wScList[\Cantor]$, they are strict. 
 
 Combining the facts that $\wScList[\Cantor]\strictlyweireducible \ScList[\Cantor]$ and $\wScList[\Cantor]\not\weireducible\UCBaire$, we conclude that $\ScList[\Cantor]\not\weireducible\UCBaire$.
 
To show that $\UCBaire\not\weireducible\ScList[\Cantor]$, we first prove that $\ScCount[\Cantor]\weireducible \codedChoice{\boldfacePi_4^0}{}{\nats}$. Given $T \in \tree_2$, let 
$$A\defas\{n: (n>0 \implies \varphi(n-1,T) \land \lnot \varphi(n,T)) \land (n=0\implies (\forall k)(\varphi(k,T))\}$$
where $\varphi(n,T)\defas(\exists \sigma_0,\dots,\sigma_{n-1})(\forall i \neq j< n)(\sigma_i \incomparable \sigma_j \land T_{\sigma_i} \in \uniquebranch_2)$. Using \thref{Pi11sets2}, it is easy to check that $\varphi(n,T)$ is $\Sigma_3^0$ and hence $A$ is a $\Pi_4^{0,T}$ subset of $\nats$. Notice that $A$ is a singleton and, by \thref{Isolated_paths}, the unique  $n \in A$ is the correct answer for $\ScCount[\Cantor]$. 

As $\ScList[\Cantor]\weireducible \wScList[\Cantor] \times \ScCount[\Cantor] \weireducible  \wScList[\Cantor]*\ScCount[\Cantor]$ we have that 
\[\firstOrderPart{\ScList[\Cantor]} \weireducible  \firstOrderPart{(\wScList[\Cantor]*\ScCount[\Cantor])} \weireducible \firstOrderPart{\wScList[\Cantor]}*\ScCount[\Cantor],\]
where the second reduction follows from \cite[Proposition 4.1(4)]{valentisolda} which states that $\firstOrderPart{(f*g)}\weireducible\firstOrderPart{f}*g$ for any $f$ and $g$. By \thref{Fop_wsclistcantor} and the fact that $\ScCount[\Cantor]\weireducible \codedChoice{\boldfacePi_4^0}{}{\nats}$ we get that $\firstOrderPart{\ScList[\Cantor]}\weireducible \codedChoice{\boldfacePi_2^0}{}{\nats}*\codedChoice{\boldfacePi_4^0}{}{\nats} \weiequiv \codedChoice{\boldfacePi_4^0}{}{\nats}$ (the last equivalence follows from \cite[Theorem 7.2]{valentisolda}).
By \thref{limdoesnotreachucbaire}(i) and (iii) and \thref{Fopcbaire},  we know that $\codedChoice{\boldfacePi_4^0}{}{\nats}\strictlyweireducible\codedChoice{\boldfacePi_1^1}{}{\nats}\weiequiv\firstOrderPart{\UCBaire}$, hence $\UCBaire\not\weireducible\ScList[\Cantor]$.

\item 
By \thref{Ucbaire_below_pk} and \thref{Pkbaire_equiv_wsclistbaire}, we have that $\UCBaire\strictlyweireducible\PK[\Baire]\weiequiv\wScList[\Baire]$. Moreover, in the proof of (ii), we showed $\ScCount[\Cantor]\weireducible \codedChoice{\boldfacePi_4^0}{}{\nats}$, which by \thref{fop_pi11} implies $\ScCount[\Cantor] \weireducible \wScList[\Baire]$. Since $\ScList[\Cantor] \weireducible \ScCount[\Cantor] \times \wScList[\Cantor]$ and $\wScList[\Cantor] \weireducible \wScList[\Baire]$ is immediate, we have that $\ScList[\Cantor] \weireducible \wScList[\Baire] \times \wScList[\Baire]$. Recalling that by \thref{Wsclist_parallelizable} $\wScList[\Baire]$ is parallelizable, we obtain that $\ScList[\Cantor] \weireducible \wScList[\Baire]$. Since $\UCBaire\weiincomparable\ScList[\Cantor]$ by the previous item, we also deduce that the reduction is strict.

By the fact that $\wScList[\Baire]\weiequiv \PK[\Baire]$ (\thref{Pkbaire_equiv_wsclistbaire}), we have that \thref{Pkbaire_below_sclistbaire} and \thref{Ucbaire_below_pk} imply respectively $\wScList[\Baire]\strictlyweireducible \ScList[\Baire]$ and $\CBaire\weiincomparable\wScList[\Baire]$. These two relationships imply that $\ScList[\Baire]\not\weireducible \CBaire$.\qedhere
\end{enumerate}
\end{proof}

\subsection{The full Cantor-Bendixson theorem}
\label{Fullcantorbendixson}

The following functions formulate the Cantor-Bendixson theorem as a problem.

\begin{definition}
Let $\X$ be a computable Polish space. We define two multi-valued functions $\multifunction{\wCB[\X]}{\negrepr{\X}}{\negrepr{\X} \times (2\times \X)^\nats}$ and $\multifunction{\CB[\X]}{\negrepr{\X}}{\negrepr{\X} \times (\nats \times (2\times \X)^\nats)}$ by
$$\wCB[\X](A) \defas \PK[\X](A) \times \wScList[\X](A) \text{ and } \CB[\X](A)\defas \PK[\X](A) \times \ScList[\X](A).$$
The multi-valued functions $\multifunction{\wCB}{\tree}{\tree \times(2\times \Baire)^\nats}$ and $\multifunction{\CB}{\tree}{\tree \times (\nats\times(2\times \Baire)^\nats)}$ are defined similarly, substituting $\PK[\X]$ with $\PK$ and $(\mathsf{w})\List[\X]$ with $(\mathsf{w})\List[\Baire]$ in the definitions above.
   \end{definition}

\begin{proposition}
\thlabel{CBtree}
$\wCB\weiequiv\CB\weiequiv\parallelization{\wf}$.
\end{proposition}
\begin{proof}

Clearly $\PK\weireducible\wCB[]\weireducible\CB[]$ and since, by \thref{Chipistrongpktree}, $\parallelization{\wf}\weiequiv\PK$ we have that $\parallelization{\wf}\weireducible\wCB[]\weireducible\CB[]$. For the opposite directions notice that $\wCB[]\weireducible\PK\times\wScList[\Baire]$ and $\CB[]\weireducible\PK\times\ScList[\Baire]$. By \thref{Pkbaire_equiv_wsclistbaire,sclist_below_pica}, we have that $\wScList[\Baire]\strictlyweireducible\ScList[\Baire]\strictlyweireducible \parallelization{\wf}$. As $\parallelization{\wf}$ is clearly parallelizable this concludes the proof. 
\end{proof}

\begin{theorem}
\thlabel{equivalencesCBbairecantor}
$\wCB[\Cantor]\weiequiv \CB[\Cantor] \weiequiv \PK[\Cantor] \weiequiv \wScList[\Baire] \weiequiv \wCB[\Baire] \weiequiv \PK[\Baire] \strictlyweireducible \ScList[\Baire]\weireducible \CB[\Baire] \strictlyweireducible \CB$.
\end{theorem}
\begin{proof}
By \thref{Pkbaire_equiv_wsclistbaire,Cantor_and_baire_same2}, $\wScList[\Baire] \weiequiv \PK[\Baire]\weiequiv \PK[\Cantor]$. By \thref{sclistcantor_summary}, $\wScList[\Cantor]\strictlyweireducible \ScList[\Cantor]\strictlyweireducible \wScList[\Baire]$. Since,  by \thref{Pkbaire_parallelizable}, $\PK[\Baire]$ is parallelizable, we obtain all the equivalences. Also, $\ScList[\Baire]\weireducible \CB[\Baire]$ is immediate, and $\PK[\Baire]\strictlyweireducible \ScList[\Baire]$ was already proven in \thref{Pkbaire_below_sclistbaire}. To prove $\CB[\Baire] \strictlyweireducible \CB$,  notice that the reduction is straightforward and $\CB[\Baire]\weireducible \wCB[\Baire]* \ScCount[\Baire] \weiequiv \wScList[\Baire]*\ScCount[\Baire]$. To conclude the proof, observe that $\parallelization{\wf}\weiequiv \CB[]$ (\thref{CBtree}) and  that  $\parallelization{\wf} \not\weireducible \wScList[\Baire]*\ScCount[\Baire]$ (see the  proof of \thref{sclist_below_pica}), hence $\CB\not \weireducible \CB[\Baire] $.
\end{proof}

The situation here is similar to the one discussed above for $\ScList[\Baire]$: we do not know if $\CBaire\strictlyweireducible \CB[\Baire]$ or $\CBaire\weiincomparable \CB[\Baire]$. It is also open whether $\ScList[\Baire]\weiequiv \CB[\Baire]$ (see \thref{questioncbaire,question:cbbaireandlist}).

In the next theorem, we explore what can be added to $\PK[\Baire]$ in order to compute $\CB[\Baire]$.
\begin{theorem}
\thlabel{AttemptCBaireCN}
$\CB[\Baire]\weireducible \ustar{\wf}\times \PK[\Baire]\weireducible \codedChoice{}{}{\nats}*\PK[\Baire]$.
\end{theorem}
\begin{proof}
For the first reduction, by \thref{equivalencesCBbairecantor}, $\PK[\Baire]\weiequiv\wCB[\Baire]$ and clearly $\CB[\Baire]\weireducible\ScCount[\Baire]\times \wCB[\Baire]$: by \thref{sclist_below_pica,Summaryfopustar}, we obtain that $\ScCount[\Baire]\weireducible \ustar{\wf}$.

For the second reduction, notice that ${\ustar{\Choice{\nats}}}\weiequiv \codedChoice{}{}{\nats}$ (\cite[Theorem 7.2]{valentisolda}). Furthermore, since $\PK[\Baire]\weiequiv\wCB[\Baire]$, $\PK[\Baire]$ is parallelizable (\thref{Pkbaire_parallelizable}) and $\PK[\Baire]\weiequiv\parallelization{\wfsierpinski}$ (\thref{Pkcantor_idPiSigma}), it suffices to show that $ \ustar{\wf} \weireducible {\ustar{\Choice{\nats}}} *\parallelization{\wfsierpinski}$.
The input for $\ustar{\wf}$ is a sequence $(T^i)_{i \in \nats} \in \tree^\nats$. Let $(p_i)_{i \in \nats}$ be a name for a solution of $\parallelization{\wfsierpinski}((T^i)_{i \in \nats})$ (recall that the only name for $0_\sierpinski$ is $0^\nats$). For every $i\in \nats$, the input for the $i$-th instance of $\codedChoice{}{}{\nats}$
\[A_i \defas \{n:(n=0\land p_i=0^\nats) \lor (n>0 \land (\exists m<n)(p_i(m)=1))\}.\]
To conclude the proof it suffices to notice that for any $i$ and $n_i \in \codedChoice{}{}{\nats}(A_i)$, $T^i \in \wellfounded$ iff $n_i\neq0$.
\end{proof}

\section{What happens in arbitrary computable metric spaces}
In this section we study the functions connected to the perfect set and Cantor-Bendixson theorems in arbitrary computable metric spaces. We start by collecting some facts about maps between spaces of closed sets. 

\label{otherspaces}
 \begin{proposition}[{\cite[proof of Proposition 3.7]{closedChoice}}]
 \thlabel{surjections}
 Let $\X$ and $\Y$ be computable metric spaces and $\partialfunction{s}{\X}{\Y}$ be a computable function with $\dom(s) \in \Pi_1^0(\X)$: then the function $\function{S}{\negrepr{\Y}}{\negrepr{\X}}$, $M\mapsto s^{-1}(M)$ is computable as well.
 \end{proposition}

\begin{definition}[{\cite{BorelComplexity}}]
 \thlabel{richness}
Given two represented spaces $\X$ and $\Y$, we say  that $\function{\iota}{\X}{\Y}$ is a \emph{computable embedding} if $\iota$ is injective and $\iota$ as well as its partial inverse $\iota^{-1}$ are computable.
If $\X$ is a computable metric space we say that $\X$ is \emph{rich} if there exists a computable embedding of $\Cantor$ into $\X$.
\end{definition}

As observed in \cite{BorelComplexity}, any  computable embedding $\function{\iota}{\Cantor}{\X}$ is such that $\range(\iota) \in \Pi_1^0(\X)$. Moreover, by \cite[Theorem 6.2]{BorelComplexity}, any perfect computable metric space $\X$ is rich.

\begin{theorem}[{\cite[Theorem 3.7]{BorelComplexity}}]
 \thlabel{embeddingtheorem}
    Let $\X$ and $\Y$ be computable metric spaces and $\function{\iota}{\X}{\Y}$ be a computable embedding with $\range(\iota) \in \Pi_1^0(\Y)$. Then the map  $\function{J}{\negrepr{\X}}{\negrepr{\Y}}, A \mapsto \iota(A)$ is computable and admits a partial computable right inverse.
\end{theorem}
 
The following is an analog of \cite[Corollaries 4.3 and 4.4]{closedChoice}.

\begin{lemma}
\thlabel{richspacesPKGeneral}
Let $\X$ and $\Y$ be computable Polish spaces and $\function{\iota}{\X}{\Y}$ be a computable embedding with $\range(\iota) \in \Pi_1^0(\Y)$. Let $f$ be any of the following: $\PST$, $(\mathsf{w})\List$, $\PK$, $(\mathsf{w})\ScList$, $(\mathsf{w})\CB$. Then $f_{\X}\weireducible f_{\Y}$. In particular, $f_{\Cantor} \weireducible f_{\Y}$ for every rich computable metric space $\Y$.
\end{lemma}
\begin{proof}
By \thref{embeddingtheorem}, the map $\function{J}{\negrepr{\X}}{\negrepr{\Y}}$ and its partial inverse are computable.  Given $A \in \dom(f_\X)$, we have that $J(A) \in \dom(f_\Y)$, as cardinality is preserved by $J$. Moreover, $J(A)$ is homeomorphic to $A$.  Depending on $f$, we use combinations of copies of $J^{-1}$ and $\iota^{-1}$ to compute from a solution for $f_\Y(J(A))$ a solution for $f_\X(A)$. For example, considering the case $f = \CB$, we have that if $(B, (n,(b_i,y_i)_{i \in \nats})) \in \CB[\Y](J(A))$ then $(J^{-1}(B), (n, (b_i,\iota^{-1}(y_i))_{i \in \nats})) \in \CB[\X](A)$.
\end{proof}
The following lemma is immediate using \thref{richspacesPKGeneral} and either \thref{cantor_and_baire_same} or \thref{Cantor_and_baire_same2} or \thref{equivalencesCBbairecantor}.
 \begin{lemma}
\thlabel{richspacesPK}
Let $f$ be any of the following: $\PST$, $\PK$, $\wCB$. For every rich computable Polish space $\X$, $f_{\Baire} \weireducible f_\X$. If moreover there exists a computable embedding $\function{\iota}{\X}{\Baire}$ with $\range(\iota)\in \Pi_1^0(\Baire)$, $f_\X\weiequiv f_{\Baire}$.
\end{lemma}

\subsection{Perfect sets}

\thref{richspacesPK} implies that $\PST[\X]\weiequiv \PST[\Baire]$ whenever $\X$ is $0$-dimensional. We can however obtain this result also for the unit interval.

\begin{theorem}
\thlabel{PST01}
$\PST[{[0,1]}]\weiequiv \PST[\Baire]$.
\end{theorem}
\begin{proof}
The right-to-left direction follows from \thref{richspacesPK}. 

For the opposite direction, since $\PST[\Baire]\weiequiv \PST[\Cantor]$ by \thref{cantor_and_baire_same}, we show instead that $\PST[{[0,1]}]\weireducible\PST[\Cantor]$.

Let $\function{\sbin}{\Cantor}{[0,1]}$ be the computable function that computes a real from its binary expansion:
\[
s_{ \mathsf{b} } (p) = \underset{ i \in \nats }{ \sum }   \frac { p(i) } { 2^{i+1} }.
\]
Notice that $\sbin$ is not injective (and hence not an embedding) as $\sbin(\sigma01^\nats) = \sbin(\sigma10^\nats)$ for any $\sigma \in \cantor$; however these are the only counterexamples to injectivity. In particular, for every $x \in [0,1]$, $\length{\sbin^{-1}(x)}\leq 2$.
For $\sigma \in \cantor$ we let $I^\sigma\defas\{x \in [0,1]:(\forall p \in \Cantor)(\sbin(p)=x \implies \sigma\sqsubset p)\}$. Notice that if $\sigma$ is not constant then $I^\sigma=(\sbin(\sigma0^\nats), \sbin(\sigma1^\nats))$, while $I^{0^n}=[0, \sbin(0^n1^\nats))$ and $I^{1^n}=(\sbin(1^n0^\nats), 1]$: all these intervals are open subsets of $[0,1]$. 

By \thref{surjections} given $A \in \negrepr{[0,1]}$ we can compute $\sbin^{-1}(A)$.

Although \thref{embeddingtheorem} does not apply, we claim that $\function{J}{\negrepr{\Cantor}}{\negrepr{[0,1]}}$, $C\mapsto \sbin(C)$, is computable (notice that, as $\Cantor$ is compact and $\sbin$ is continuous, the image of a closed set is closed). To prove that $J$ is computable, we proceed as follows: let $S \in \tree_2$ be a name for $C \in \negrepr{\Cantor}$, i.e.\ a tree such that $\body{S}=C$. Recall that, by \thref{Complexityresults}(i), $\wellfounded_2$ is a $\Sigma_1^0$ set. We compute $B \in \negrepr{[0,1]}$ as follows:   
\begin{enumerate}[(i)]
    \item whenever we witness that ${S_\sigma} \in \wellfounded_2$,  we list $I^\sigma$ in the complement of $B$;
   \item whenever we witness that ${S_{\sigma01^i}}$ and ${ S_{\sigma10^i}}$ are in $\wellfounded_2$ for some $i \in \nats$, we list in the complement of $B$ the open interval $(\sbin(\sigma01^i0^\nats),\sbin(\sigma10^i1^\nats))$ which coincides with $I^{\sigma01^i} \cup I^{\sigma10^i} \cup \{\sbin(\sigma01^\nats)\}$.
\end{enumerate}
We need to check that $B=J(C)$, i.e.\ for every $x \in [0,1]$, $x \notin B$ iff $\sbin^{-1}(x)\cap C = \emptyset$. 

If $x \notin B$ because $x \in I^\sigma$ for some $\sigma$ with ${S_\sigma} \in \wellfounded_2$, then $\sigma$ is a prefix of every element of $\sbin^{-1}(x)$ and hence $\sbin^{-1}(x)\cap C = \emptyset$. If $x \notin B$ because $x \in (\sbin(\sigma01^i0^\nats),\sbin(\sigma10^i1^\nats))$ for some $\sigma$ and $i$, then either $x \in I^{\sigma01^i} \cup I^{\sigma10^i}$, in which case we can apply the previous argument to one of $\sigma01^i$ and $\sigma10^i$, or $x = \sbin(\sigma01^\nats) = \sbin(\sigma10^\nats)$; in this case we know that both $\sigma01^\nats$ and $\sigma10^\nats$ do not belong to $C$.

For the converse, consider first the case where $\sbin^{-1}(x) = \{q\}$ and $x \notin \{0,1\}$: then $q$ is not eventually constant. Since $q \notin C$, there exists $\sigma\sqsubset q$ such that $\sigma \notin S$ and hence $I^\sigma$ is listed in the complement of $B$. As $q \notin \{\sigma0^\nats, \sigma1^\nats\}$, we obtain that $x \in I^\sigma$ and hence $x \notin B$. The case in which $x \in \{0,1\}$ is analogous. 
If  ${\sbin^{-1}(x)} = \{q_0,q_1\}$ then, as noticed above, there exists $\tau$ such that $q_0 = \tau01^\nats$ and $q_1=\tau10^\nats$. Since $q_0,q_1 \notin C$ we have $\tau01^i, \tau10^i \notin S$ for some $i$. Then $x \in (\sbin(\tau01^i0^\nats),\sbin(\tau10^i1^\nats))$, and this interval is listed in the complement of $B$ by condition (ii). Therefore, $x \notin B$.

We now describe the reduction. Given an uncountable $A\in \negrepr{[0,1]}$, we can compute $\sbin^{-1}(A) \in \negrepr{\Cantor}$  which is uncountable as well. Let $P \in \PST[\Cantor](\sbin^{-1}(A))$ and $B = J(P)$. It suffices to show that $B \subseteq A$ and that $B$ is perfect. 

If $x \in B$ then there exists $q \in \sbin^{-1}(x) \cap P$. Since $P \subseteq \sbin^{-1}(A)$ we get that $\sbin(q)=x \in A$. This shows that $B \subseteq A$.

 It remains to show that $B$ is perfect. Suppose not: then there exists $x \in B$ and some open interval $I \subseteq [0,1]$ such that $I \cap B=\{x\}$. By continuity of $\sbin$, $\sbin^{-1}(I \cap B)$ is an open set in $P$ which has at most two members; these points are isolated in $P$, contradicting the perfectness of $P$. 
\end{proof}

\begin{remark}
Following the ideas of the previous proof and using some extra care it is possible to prove that $\PST[\Baire]\weiequiv \PST[\mathbb{R}]$: replace $s_{ \mathsf{b} }$ with $\function{s_{ \mathsf{b} }'}{\nats\times\Cantor}{\mathbb{R}}$ defined by
\[
s_{ \mathsf{b} }' (n,p) = (-1)^{n} \cdot \left\lceil\frac{n}{2}\right\rceil +  s_{ \mathsf{b} }(p).
\]
\end{remark}

We do not know whether there exist rich computable Polish spaces $\X$ such that $\PST[\Baire] \strictlyweireducible \PST[\X]$ (see \thref{question:richspacesBaire}).

\subsection{(Weak) lists}
The following classical fact helps in the next proofs.
\begin{theorem}[{\cite[Theorem 3E.6]{Moschovakis}}]
	\thlabel{theorem3e6moschovakis}
	For every computable metric space $\X$ there is a computable surjection $\function{s}{\Baire}{\X}$ and $A \in \Pi_1^0(\Baire)$ such that $s$ is one-to-one on $A$ and $s(A)=\X$.
\end{theorem}

\begin{lemma}
\thlabel{listupperbound}
Let $\X$ be a computable metric space.  Then $(\mathsf{w})\List[\X]\weireducible( \mathsf{w})\List[\Baire]$.
\end{lemma}
\begin{proof}
We prove only $\List[\X]\weireducible\List[\Baire]$ as the other reduction is similar. Let $s$ and $A$ be as in \thref{theorem3e6moschovakis} and $s_A$ be the restriction of $s$ to $A$. By \thref{surjections}, the function $\function{S_A}{\negrepr{\X}}{\negrepr{\Baire}}$ such that $S_A(M)=s_A^{-1}(M)$ is computable: hence given $C \in \negrepr{\X}$ and $(n,(b_i,p_i)_{i \in \nats}) \in \List[\Baire](S_A(C))$ we have that $(n,(b_i,s(p_i))_{i \in \nats}) \in \List[\X](C)$.
\end{proof}

\begin{lemma}
\thlabel{richspaceslist}
Let $\X,\Y$ be computable metric spaces and $\function{\iota}{\X}{\Y}$ be a computable embedding with $\range(\iota) \in \Pi_1^0(\Y)$. Then $(\mathsf{w})\List[\X] \weireducible (\mathsf{w})\List[\Y]$. In particular, $(\mathsf{w})\List[\Cantor] \weireducible (\mathsf{w})\List[\Y]$ for every rich computable metric space $\Y$.
\end{lemma}
\begin{proof}
We only prove that $\List[\X]\weireducible\List[\Y]$, the other reduction is similar.  By \thref{embeddingtheorem} the map $\function{J}{\negrepr{\X}}{\negrepr{\Y}}$ is computable. Given $A \in \dom(\List[\X])$ we have that $J(A) \in \dom(\List[\Y])$: moreover, given $(n,(b_i,p_i))_{i \in \nats} \in \List[\Y](J(A))$ we have that  $(n,(b_i,\iota^{-1}(p_i)))_{i \in \nats} \in \List[\X](A)$. 
\end{proof}

\begin{lemma}
\thlabel{wlistequivelenceotherspaces}
$\wList[\mathbb{R}]\weiequiv \wList[{[0,1]}]\weiequiv \wList[\Cantor]$.
\end{lemma}
\begin{proof}
The fact that $\wList[{[0,1]}]\weireducible\wList[\mathbb{R}]$ is immediate and $\wList[\Cantor]\weireducible \wList[{[0,1]}]$ follows from \thref{richspaceslist}. 

Recall that  $\wList[\Cantor]$ is parallelizable (\thref{wsclistcantorparallelizable}) and notice that it is straightforward to check that $\wList[\mathbb{R}] \weireducible \parallelization{\wList[{[0,1]}]}$. Hence, it suffices to show that $\wList[{[0,1]}]\weireducible\wList[\Cantor]$. The function $\sbin$ of the proof of \thref{PST01} is useful also here. Consider an input $A \in \negrepr{[0,1]}$ and let $(b_i,p_i)_{i \in \nats}\in \wList[\Cantor](\sbin^{-1}(A))$: it is straightforward to check that $(b_i,\sbin(p_i))_{i \in \nats}$ is a solution of $\wList[{[0,1]}](A)$.
\end{proof}

Notice that the argument above shows that $\wList[\X]\weireducible\wList[\Cantor]$ for every computable metric space $\X$ such that there exists an admissible representation $\partialfunction{\delta}{\Cantor}{\X}$ with $\dom (\delta) \in \Pi_1^0(\Cantor)$ and such that $\length{\delta^{-1}(x)}\leq \aleph_0$ for every $x \in \X$. In particular $\wList[ {[0,1]^d}] \weiequiv \wList[\Cantor]$ for any $d \in \nats$.

Notice that the situation for $\List[\X]$ is less clear: for example, we do not know if, in contrast to what happens for $\wList[\X]$, $\List[\Cantor] \strictlyweireducible \List[\mathbb{R}]$ (see \thref{question:list}).

We now consider listing problems on countable spaces. Let us start from finite spaces: for $n>0$, we denote by $\mathbf{n}$ the space consisting of $\{0,\dots,n-1\}$ with the discrete topology and an arbitrary computable metric, which is obviously a computable metric space. 
\begin{proposition}
\thlabel{FiniteListSame}
For every $n >0$, $\wList[\mathbf{n}]\weiequiv \List[\mathbf{n}] \weiequiv \lpo^{n}$ and therefore $\List[\mathbf{n}]\strictlyweireducible \List[\mathbf{n}+1]$.
\end{proposition}
\begin{proof}
 The fact that $\wList[\mathbf{n}]\weireducible \List[\mathbf{n}]$ is trivial. For the converse let $A\in \negrepr{\mathbf{n}}$ and let $(b_i,m_i)_{i \in \nats} \in \wList[\mathbf{n}](A)$. Notice that for every $m<n$  exactly one of $m \notin A$ and $(\exists i) (b_i=1 \land m_i=m)$ holds: since both conditions are $\Sigma_1^{0}$ we can compute whether $m \in A$ or not. This allows us to compute $\length{A}$ and, together with $(b_i,m_i)_{i \in \nats}$ we obtain a name for $\List[\mathbf{n}]$ (see \thref{equivalentcersionslists}).

To show that $\wList[\mathbf{n}]\weireducible \lpo^{n}$, let $A \in \negrepr{\mathbf{n}}$ and fix a computable formula $\varphi$ such that $i \in A$ iff $(\forall k) \varphi(i,k,A)$. The input $p_i \in \Cantor$ for the $i$-th instance of $\lpo$ is defined by $p_i(k)=1$ iff $\lnot \varphi(i,k,A)$, so that $\lpo(p_i)=1 \iff i \in A$. For all $i \in \nats$ define 
\[
b_i \defas \begin{cases}
1 & \text{if } i<n \text{ and } \lpo(p_i)=1;\\
0 & \text{otherwise.}
\end{cases}
\qquad 
x_i \defas \begin{cases}
i & \text{if } i<n;\\
0 & \text{otherwise.}
\end{cases}
\]
Then, $(b_i,x_i)_{i \in \nats} \in \wList[\mathbf{n}](A)$.

For the opposite direction, we show that $\lpo^{n}\weireducible \wList[\mathbf{n}]$. Let $(p_j)_{j < n}$ be an input for $\lpo^{n}$. Consider $A \defas \{j < n: p_j=0^\nats\} \in \negrepr{\mathbf{n}}$ and let $(b_i,m_i)_{i \in \nats} \in \List[\mathbf{n}](A)$. Notice that, for every $j<n$, $p_j = 0^\nats$ iff $(\exists i) (b_i=1 \land m_i = j)$. We thus can compute $\lpo(p_j)$ by searching for $i$ such that either $p_j(i)=i$ or $b_i=1$ and $m_i=j$. 

The fact that $\List[\mathbf{n}]\strictlyweireducible \List[\mathbf{n}+1]$ follows from  $\lpo^{n}\strictlyweireducible \lpo^{n+1}$ (\cite[Corollary 6.7]{BG09}).
\end{proof}

We say that a computable metric space is \emph{effectively countable} if there exists a computable surjection $\function{f}{\nats}{\X}$.
\begin{lemma}
\thlabel{effectivelycountable}
For any computable metric space $\X$ which is effectively countable, $\wList[\X]\weireducible\mflim$.
\end{lemma}
\begin{proof}
Fix a computable surjection $\function{f}{\nats}{\X}$. Recalling that $\mflim \weiequiv \parallelization{\lpo}$, the proof is a straightforward generalization of the proof of $\wList[\mathbf{n}]\weireducible \lpo^{n}$ in \thref{FiniteListSame}.
\end{proof}

We say that a computable metric space $\X$ is \emph{effectively infinite} if there exists a computable sequence $(U_i)_{i \in \nats}$ of open sets in $\X$ such that $(\forall i)(U_i \not\subseteq \underset{j\neq i}{\bigcup}  U_j)$.

\begin{lemma}
\thlabel{effectivelyinfinite}
For every countable computable metric space $\X$ which is effectively infinite, $\mflim \weireducible \wList[\X]$.
\end{lemma}
\begin{proof}
Fix a sequence $(U_i)_{i \in \nats}$ witnessing that $\X$ is effectively infinite. Recall that $\mflim \weiequiv \parallelization{\lpo}$ so that it suffices to show $\parallelization{\lpo}\weireducible\wList[\X]$. 

Given an input $(p_i)_{i \in \nats}$ for $\parallelization{\lpo}$, let $A\defas\{x \in \X: (\forall i)(x \in U_i \implies p_i=0^\nats)\} \in \negrepr{\X}$ and notice  that $\length{A} \leq \aleph_0$ because $\X$ is countable. Notice that $p_i = 0^\nats$ iff $A \cap U_i \neq \emptyset$ (for the forward direction use the existence of $y_i \in U_i$ such that $y_i \notin \underset{j\neq i}{\bigcup} U_j$ by definition of effectively infinite).

Fix $(b_i,x_i)_{i \in \nats} \in \wList[\X](A)$. By the above observation, for every $i \in \nats$ we get
$$p_i=0^\nats \iff  (\exists k)(b_k=1\land x_k \in U_i).$$
Since we showed the equivalence of the $\Pi_1^0$ condition $p_i=0^\nats$ with a $\Sigma_1^0$ condition, we can compute $\lpo(p_i)$ for every $i \in \nats$.
\end{proof}

Many natural countable computable metric spaces, not necessarily Polish, are easily seen to be both effectively countable and effectively infinite. We thus can combine \thref{effectivelycountable,effectivelyinfinite} to obtain $\wList[\X] \weiequiv \mflim$ for several countable spaces, both compact and non-compact:

\begin{corollary}
\thlabel{corollarywListnats}
$\wList[\nats] \weiequiv \wList[\mathcal{K}] \weiequiv \wList[\mathbb{Q}] \weiequiv \mflim$, where $\mathcal{K} = \{0\} \cup \{2^{-n}: n \in \nats\} $.
\end{corollary}

\begin{proposition}
\thlabel{wlistNparallelizable}
$\wList[\nats]\strictlyweireducible \List[\nats]$.
\end{proposition}
\begin{proof}

The fact that $\wList[\nats]\weireducible\List[\nats]$ is trivial. For strictness, we show that $\lpo'\weireducible \List[\nats]$: since $\lpo'\weiincomparable\mflim$ and $\wList[\nats]\weiequiv\mflim$ by \thref{corollarywListnats} this suffices to conclude the proof.

We can think of $\lpo'$ as the function that, given in input $p \in \Cantor$, is such that $\lpo'(p)=1 \iff (\exists^\infty i)(p(i)=0)$. For any $p \in \Cantor$, let $A\defas \{i:p(i)= 0\}$: given $(n,(b_i,p_i)_{i \in \nats}) \in \List[\nats](A)$ it is clear that $\lpo'(p)=1$ iff $n=0$. 
\end{proof}

\subsection{The Cantor-Bendixson theorem}
Notice that it makes sense to study $\PK[\X]$ only when $\X$ is an uncountable effectively Polish space: indeed, if $\X$ is countable, $\PK[\X]$ is the function with constant value $\emptyset$.
\begin{lemma}
	\thlabel{pi11countabilityinotherspaces}
	For any computable Polish space $\X$ the set $\{C \in \negrepr{\X}:\length{C} \leq \aleph_0\}$ is $\Pi_1^1$.
\end{lemma}
\begin{proof}
	Let $s$ and $A$ be as in \thref{theorem3e6moschovakis} and denote by $s_A$ be the restriction of $s$ to $A$. By \thref{surjections}, the function $\function{S}{\negrepr{\X}}{\negrepr{\Baire}}$ defined by $S(C)=s_A^{-1}(C)$ is computable. Since $s_A$ is a bijection, we obtain that  $\length{C}=\length{S(C)}$. Recall from \S \ref{representedspaces} that we can represent $S(C)$ via some $T \in \tree$ such that $S(C)=\body{T}$. To conclude the proof notice that $\length{S(C)} \leq \aleph_0$ iff $T \in \countable$ and, by \thref{Complexityresults}(ii), $\countable$ is a $\Pi_1^1$ set. 
\end{proof}

Recall  that in  \S \ref{representedspaces} we fixed an enumeration $(B_i)_{i \in \nats}$ of all basic open sets of $\X$,  where the ball $B_{\pairing{n,m}}$ is centered in $\alpha(n)$ and has radius $q_m$.
\begin{theorem}
	\thlabel{perfectkernelforallx}
	For every rich computable Polish space $\X$, $\PK[\X]\weiequiv \PK[\Baire]$.
\end{theorem}
\begin{proof}
The right-to-left direction is \thref{richspacesPK}. For the converse reduction, by  \thref{Pkcantor_idPiSigma} we have that $\PK[\Baire]\weiequiv \parallelization{\wfsierpinski}$, hence it suffices to show that $\PK[\X]\weireducible\parallelization{\wfsierpinski}$. Let $C \in \negrepr{\X}$ be an input for $\PK[\X]$. 

Notice that $\length{B_{\str{n,m}} \cap C}\leq \aleph_0$ iff $(\forall \epsilon>0)(\length{ \{x \in \X:d(x,\alpha(n))\leq q_m-\epsilon\}\cap C}\leq \aleph_0)$: as $\{x \in \X:d(x,\alpha(n))\leq q_m - \epsilon\}\cap C$ is a closed set that can be uniformly computed from $C$, $n$ and $m$, by \thref{pi11countabilityinotherspaces} we get that $\length{B_{\str{n,m}}\cap C}\leq \aleph_0$ is $\Pi_1^1$.

We can therefore compute a sequence $(T^{\pairing{n,m}})_{n,m\in \nats}$ of trees  such that $T^{\pairing{n,m}} \in \wellfounded$ iff $\length{B_{\str{n,m}} \cap C}\leq \aleph_0$. Hence, searching the output of $\parallelization{\wfsierpinski}((T^{\pairing{n,m}})_{n,m \in \nats})$ for the $\pairing{n,m}$'s such that  $\wfsierpinski(T^{\str{n,m}})=1$, we eventually enumerate all the $B_{\str{n,m}}$ such that $B_{\str{n,m}} \cap \PK[\X](C)=\emptyset$, thus obtaining a name for $\PK[\X](C)\in \negrepr{\X}$.
\end{proof}

The proof of the next Lemma combines ideas from the proof of \thref{wsclistreduciblepkbaire} and \thref{listupperbound}.

\begin{lemma}
\thlabel{sclistupperbound}
For every computable Polish space $\X$, $\wScList[\X]\weireducible \wScList[\Baire]$.
\end{lemma}
\begin{proof}
We show that $\wScList[\X]\weireducible \parallelization{\wfsierpinski} \times \parallelization{\completion{\wList[\X]}}\weireducible \wScList[\Baire]$. 

The first reduction is obtained by generalizing the proof of \thref{wsclistreduciblepkbaire} (which is the case $\X= \Baire$): given $C \in \negrepr{\X}$ it suffices to use as input for the $\pairing{n,m}$-th instances of $\wfsierpinski$ and $\completion{\wList[\X]}$ respectively a tree $T^{\str{n,m}}$ such that $T^{\str{n,m}} \in \wellfounded$ iff $(\forall \epsilon>0) (\length{ \{x \in \X : d(x,\alpha(n))\leq q_m - \epsilon\}\cap C}\leq \aleph_0)$ (see the proof \thref{perfectkernelforallx}) and $\{x \in \X:d(x,\alpha(n))\leq q_m-\epsilon\}\cap C$.

For the second reduction, notice that $\completion{\wList[\X]}\weireducible \completion{\wList[\Baire]}$ by essentially the same proof of \thref{listupperbound}. As $\wScList[\Baire]$ is parallelizable (\thref{Wsclist_parallelizable}) and $\completion{\wList[\Baire]}\weireducible\wScList[\Baire]\weiequiv \parallelization{\wfsierpinski}$ (\thref{Pkcantor_idPiSigma,Pkbaire_equiv_wsclistbaire} and \thref{completionofwlist_below_pk}) we obtain the reduction.
\end{proof}

The same proof of \thref{richspaceslist} yields the following Lemma.

\begin{lemma}
\thlabel{richspacessclist}
Let $\X$ and $\Y$ be computable metric spaces and $\function{\iota}{\X}{\Y}$ be a computable embedding with $\range(\iota) \in \Pi_1^0(\Y)$. Then $(\mathsf{w})\ScList[\X] \weireducible (\mathsf{w})\ScList[\Y]$. In particular, $(\mathsf{w})\ScList[\Cantor] \weireducible (\mathsf{w})\ScList[\Y]$ for every rich computable metric space $\Y$.
\end{lemma}

Recall that the problems $\PK[\X]$, where $\X$ is a rich computable Polish space,  are all Weihrauch equivalent (\thref{perfectkernelforallx}). Combining \thref{richspacessclist,sclistupperbound}, we obtain $\wScList[\Cantor] \weireducible \wScList[\X] \weireducible \wScList[\Baire]$, but we do not know whether for some rich computable Polish space $\X$ both reductions are strict (see \thref{question:sclistcantor}).

\begin{theorem}
\thlabel{wsclistxequivwsclistbaire}
For any rich computable Polish space $\X$, $\wCB[\X]\weiequiv\PK[\Baire]$.
\end{theorem}
\begin{proof}
For the left-to-right reduction notice that $\wCB[\X]\weireducible \PK[\X] \times \wScList[\X]$. By \thref{perfectkernelforallx} and \thref{sclistupperbound} we know that $\PK[\X]\weiequiv \PK[\Baire]$ and $\wScList[\X]\weireducible\wScList[\Baire]$. Since  $\wScList[\Baire]\weiequiv \PK[\Baire]$ (\thref{Pkbaire_equiv_wsclistbaire}) and  $\PK[\Baire]$ is parallelizable (\thref{Pkbaire_parallelizable}), this concludes the reduction. The other direction follows from the combination of \thref{perfectkernelforallx} and the fact that $\PK[\X]\weireducible \wCB[\X]$.
\end{proof}

In the literature, there are many equivalent definitions of  computably compact represented spaces. The following is the most convenient for our purposes.

\begin{definition}[{\cite[\S 5]{topaspects}}]
\thlabel{compactDefinition}
A subset $K$ of a represented space $\X$ is computably compact if $\{A \in \negrepr{\X} : A \cap K = \emptyset\}$ is $\Sigma_1^0$.
\end{definition}

\begin{definition}
A computable metric space $\X$ is computably $K_\sigma$ if there exists a computable sequence $(K_i)_{i \in \nats}$ of nonempty computably compact sets with $\X=\bigcup_{i \in \nats} K_i$.
\end{definition}

The following remark extends \thref{Pi11sets2} to computably $K_\sigma$ spaces.

\begin{remark}
\thlabel{ComplexityresultsKsigmaX}
Let $\X$ be a computably $K_\sigma$ space and let $(K_i)_{i \in \nats}$ witness this property.

Notice that for $C\in \negrepr{\X}$, $C=\emptyset$ iff $(\forall i)(K_i \cap C=\emptyset)$, i.e.\ a  $\Pi_2^0$ condition. Moreover, $C\cap B_{\pairing{n,m}}=\emptyset $ iff
$(\forall k) (\{x \in \X:d(x,\alpha(n))\leq q_m-2^{-k}\}\cap C=\emptyset)$, so that this condition is $\Pi_2^0$ as well. 

Now, $\length{C}=1$ iff $C \neq \emptyset$ and 
\[(\forall n,n',m,m')(d(\alpha(n),\alpha(n'))\geq q_m+q_{m'} \implies B_{\str{n,m}} \cap C=\emptyset \lor B_{\str{n',m'}} \cap C =\emptyset)\]
is $\Pi_2^0$. Now $\length{C\cap B_{\str{n,m}}}=1$ is the conjunction of a $\Sigma_2^0$ and a $\Pi_2^0$ formula because it is equivalent to $C\cap B_{\pairing{n,m}} \neq \emptyset$ and
\begin{align*}
(\forall n',n'',m',m'')&(d(\alpha(n),\alpha(n')) \leq q_m-q_{m'} \land  d(\alpha(n),\alpha(n'')) \leq q_m-q_{m''} \land  \\    
 & \land d(\alpha(n'),\alpha(n''))\geq q_{m'}+q_{m''} \implies B_{\str{n',m'}} \cap C=\emptyset \lor B_{\str{n'',m''}} \cap C =\emptyset)\}.
\end{align*}

\end{remark}

\begin{lemma}
\thlabel{Ksigma}
	For every rich computable Polish computably $K_\sigma$ space $\X$, $\CB[\X]\weiequiv \PK[\Baire]$.
\end{lemma}
\begin{proof}
The right to left direction follows from the facts that $\PK[\X]\weireducible\CB[\X]$ and that, by \thref{perfectkernelforallx}, $\PK[\X]\weiequiv \PK[\Baire]$.

For the opposite direction, notice that $\CB[\X]\weireducible \wCB[\X]\times \ScCount[\X]$. Since $\PK[\Baire]$ is parallelizable (\thref{Pkbaire_parallelizable}) and $\PK[\Baire]\weiequiv \wCB[\X]$ (\thref{wsclistxequivwsclistbaire}), it suffices to show that $\ScCount[\X]\weireducible\PK[\Baire]$. To do so, we now adapt the proof of \thref{sclistcantor_summary}(ii) to show that $\ScCount[\X]\weireducible \codedChoice{\boldfacePi_4^0}{}{\nats}$: this concludes the proof as $\codedChoice{\boldfacePi_4^0}{}{\nats} \weireducible \codedChoice{\boldfacePi_1^1}{}{\nats} \weiequiv \firstOrderPart{\PK[\Baire]}$ (\thref{fop_pi11}). Given in input $C \in \negrepr{\X}$, let 
\[ A\defas \{k: (k>0 \implies \varphi(k-1,C) \land \lnot \varphi(k,C)) \land (k=0\implies (\forall m)(\varphi(m,C)))\},\]
 where $\varphi(k,C)$ says that there exists a finite string $\sigma = (\pairing{n_0,q_0},\hdots,\pairing{n_{k-1},q_{k-1}})\in \nats^k$ such that for every $i\neq j <k$,
\[d(\alpha(n_i),\alpha(n_j))\geq q_{i}+q_{j} \land \length{C\cap B_{\pairing{n_i,q_i}}}=1. \]

By \thref{ComplexityresultsKsigmaX}, it is easy to check that each $\varphi$ is $\Sigma_3^0$ and hence $A$ is $\Pi_4^{0}$. By \thref{Isolated_paths} the unique $k \in A$ is the correct answer for $\ScCount[\X](C)$. 
\end{proof}

The final part of this section is devoted to spaces that are not $K_\sigma$.

Recall the following consequence of Hurewicz's theorem from classical descriptive set theory.

\begin{theorem}[{\cite[Theorem 7.10]{kechris2012classical}}]
\thlabel{hurewicz}
Let $\X$ be a Polish space. Then there is an embedding $\function{\iota}{\Baire}{\X}$ such that $\range(\iota)$ is closed iff $X$ is not $K_\sigma$.
\end{theorem}

\begin{definition}
We say that a computable Polish space is computably non-$K_\sigma$ if there exists a computable embedding $\function{\iota}{\Baire}{\X}$ such that $\range(\iota) \in \Pi_1^0(\X)$.
\end{definition}

The following theorem is a corollary  of \thref{richspacesPKGeneral}.
\begin{theorem}
For any rich computable Polish  space $\X$ which is computably non-$K_\sigma$, $\CB[\Baire]\weireducible \CB[\X]$.
\end{theorem}

We leave open the question of whether there is a rich computable Polish space $\X$ such that $\CB[\Baire] \strictlyweireducible \CB[\X]$ (see \thref{question:sclistcbbaire}).

\section{Open problems}
\label{Openquestions}
In this paper, we studied problems related to the Cantor-Bendixson theorem. In contrast with reverse mathematics, we showed that many such problems lie in different Weihrauch degrees; some of these problems still lack a complete classification.

In \thref{Pk_below__chipi}, we showed that $\PK[\Baire] \strictlyweireducible \PK \weireducible \mflim * \PK[\Baire]$. Upon hearing about this result, Linda Westrick asked the following question.

\begin{open}
\thlabel{PKandLimQuestion}
Is it true that $\PK \weiequiv \mflim * \PK[\Baire]$?
\end{open}

By \thref{Pkbaire_below_sclistbaire,sclist_below_pica,Summaryfopustar}, and the fact that $\ScCount[\Baire]$ is a first-order problem, we obtain $\wf^* \weireducible \ScCount[\Baire] \weireducible  \firstOrderPart{\parallelization{\wf}} \weiequiv \ustar{\wf}$. By \cite[Corollary 7.8]{valentisolda}, $\wf^* \strictlyweireducible \ustar{\wf}$ and therefore at least one of the inequalities is strict.

\begin{open}
\thlabel{question:sccount}
Characterize the Weihrauch degree of $\ScCount[\Baire]$. 
\end{open}
 In particular, if $\ScCount[\Baire] \weiequiv \firstOrderPart{\parallelization{\wf}}$ we would obtain a nice characterization of the first-order part of $\parallelization{\wf}$.
  
A related question is the following.

\begin{open}
\thlabel{question:cbbaireandlist}
Is it true that $\CB[\Baire] \weireducible \ScList[\Baire]$ (and hence  $\CB[\Baire] \weiequiv \ScList[\Baire]$)?
\end{open}
Notice that we proved that $\ScList[\Cantor] \strictlyweireducible \CB[\Cantor]$ and $\ScList[\Baire]\strictlyweireducible \CB$ (\thref{equivalencesCBbairecantor,sclistcantor_summary}(iii)). A negative answer to \thref{question:cbbaireandlist} would confirm this pattern. However, we have $\PK[\Baire] \strictlyweireducible \ScList[\Baire]$, while $\ScList[\Cantor] \strictlyweireducible \PK[\Cantor]\weiequiv \CB[\Cantor]$ and $\ScList[\Baire]\strictlyweireducible \PK \weiequiv \CB$ (\thref{equivalencesCBbairecantor,CBtree,sclistcantor_summary}(iii)): therefore the situation in $\Baire$ differs from those in $\Cantor$ and $\tree$ and a positive answer is possible. In this case, we would obtain an unexpected result: namely, that the gap between $\PK[\Baire]$ and $\CB[\Baire]$ is due entirely to the scattered part and its cardinality, rather than to the perfect kernel. If this is the case, the cardinality of the scattered part (i.e.\ $\ScCount[\Baire]$) would be of crucial importance because the scattered part on its own is not enough as witnessed by the fact that $\wScList[\Baire] \weiequiv \PK[\Baire] \strictlyweireducible \CB[\Baire]$ (\thref{Pkbaire_equiv_wsclistbaire,equivalencesCBbairecantor}). 

The following questions are strictly related and concern the relationship of two of our problems with $\CBaire$, which plays a major role in the Weihrauch lattice. Choice principles have a convenient definition and hence, it is quite natural to compare any problem with them. In particular, $\CBaire$ plays a pivotal role among  the problems that, from the point of view of reverse mathematics, are equivalent to $\mathsf{ATR}_0$.

\begin{open}
\thlabel{questioncbaire}
Is it true that $\CBaire \weireducible\CB[\Baire]$? Even more, does $\CBaire \weireducible \ScList[\Baire]$ hold?
\end{open}

By \thref{Ucbaire_below_pk,CBtree,AttemptCBaireCN} we obtain $\CBaire \not\weireducible \PK[\Baire] \strictlyweireducible \CB[\Baire] \weireducible \codedChoice{}{}{\nats}* \PK[\Baire]$. Since $\ScList[\Baire] \weireducible \CB[\Baire]$ this implies that to answer negatively both questions it suffices to show that $\CBaire \not\weireducible \codedChoice{}{}{\nats} * \PK[\Baire]$.  By \cite[Theorem 7.11]{closedChoice} we know that $f \weireducible \codedChoice{}{}{\nats}$ iff $f$ is computable with finitely many mind-changes. In other words, $\CBaire \weireducible \codedChoice{}{}{\nats}* \PK[\Baire]$ iff  $\CBaire$ can be reduced to $\PK[\Baire]$ employing a backward functional which is computable with finitely many mind-changes: intuitively, this seems unlikely to hold.

The last section left open some interesting questions. First of all, by \thref{richspacesPK} we have that $\PST[\Baire]$ is a lower bound for $\PST[\X]$ whenever $\X$ is a rich computable Polish space. In \thref{PST01} we showed that equivalence holds when $\X = [0,1]$ or $\X = \mathbb{R}$. 
\begin{open}
\thlabel{question:richspacesBaire}
Is there a rich computable Polish space $\X$ such that $\PST[\Baire] \strictlyweireducible \PST[\X]$?
\end{open}

Concerning the listing problems for countable closed sets, the situation for the so-called weak lists is quite clear, while we do not have a satisfactory result for problems requiring also the cardinality of the set. An open question is the following:

\begin{open}
\thlabel{question:list}
Does $\List[\Cantor] \strictlyweireducible \List[\mathbb{R}]$?
\end{open}

By \thref{perfectkernelforallx,wsclistxequivwsclistbaire} all problems of the form $\PK[\X]$ and $\wCB[\X]$ belong to the same Weihrauch degree as long as $\X$ is a rich computable Polish space. In contrast, we do not know if the same happens with $\CB[\X]$.

\begin{open}
\thlabel{question:sclistcbbaire}
Is there a rich computable Polish space $\X$ such that $\CB[\Baire] \strictlyweireducible \CB[\X]$? By \thref{Ksigma} if such $\X$ exists must be computably non-$K_\sigma$. This problem is strictly related to the existence of a rich computable Polish space $\X$ such that $\ScList[\Baire] \strictlyweireducible \ScList[\X]$.
\end{open}

The last problem concerns the weak form of listing the scattered part of a set.

\begin{open}
\thlabel{question:sclistcantor}
Is there a rich computable Polish space $\X$ such that $\wScList[\Cantor] \strictlyweireducible \wScList[\X] \strictlyweireducible \wScList[\Baire]$?
\end{open}

\providecommand{\bysame}{\leavevmode\hbox to3em{\hrulefill}\thinspace}
\providecommand{\MR}{\relax\ifhmode\unskip\space\fi MR }
\providecommand{\MRhref}[2]{%
  \href{http://www.ams.org/mathscinet-getitem?mr=#1}{#2}
}
\providecommand{\href}[2]{#2}

\end{document}